\documentclass[reqno,12pt]{amsart}
\usepackage[utf8]{inputenc}
\usepackage{algorithmic}
\usepackage{amsmath}
\usepackage{amssymb}
\usepackage{amsthm}
\usepackage{amsrefs}
\usepackage{caption}
\usepackage{changepage}
\usepackage[usenames,dvipsnames]{color}
\usepackage{hyperref}
\hypersetup{colorlinks=true,urlcolor=blue,citecolor=blue,linkcolor=blue}
\usepackage{tikz-cd}
\usepackage{fullpage}
\usepackage{colonequals}

\newcommand{\calH}{\mathcal{H}}
\newcommand{\Hecke}{\calH}
\newcommand{\Q}{\mathbb{Q}}
\newcommand{\Z}{\mathbb{Z}}
\newcommand{\R}{\mathbb{R}}
\newcommand{\C}{\mathbb{C}}

\newcommand{\Lat}{\Lambda}

\newcommand{\modform}[4]{\textup{\href{https://www.lmfdb.org/ModularForm/GL2/Q/holomorphic/#1/#2/#3/#4/}{\color{blue}{\textsf{#1.#2.#3.#4}}}}}
\newcommand{\hmf}[4]{\href{https://www.lmfdb.org/ModularForm/GL2/TotallyReal/#1/holomorphic/#1-#2.#3-#4}{\color{blue}{\textsf{#1-#2.#3-#4}}}}

\newcommand{\frakp}{\mathfrak{p}}
\newcommand{\frakq}{\mathfrak{q}}

\DeclareMathOperator{\Asai}{Asai}

\DeclareMathOperator{\Cls}{Cls}
\DeclareMathOperator{\Cl}{Cl}

\DeclareMathOperator{\DN}{DN}
\DeclareMathOperator{\Gal}{Gal}
\DeclareMathOperator{\Gen}{Gen}
\DeclareMathOperator{\gen}{\Gen}
\DeclareMathOperator{\genus}{\Gen}
\DeclareMathOperator{\GL}{GL}

\DeclareMathOperator{\M}{M}
\DeclareMathOperator{\Mat}{Mat}
\DeclareMathOperator{\Nm}{Nm}
\DeclareMathOperator{\OO}{O}
\DeclareMathOperator{\nrd}{nrd}

\DeclareMathOperator{\SL}{SL}
\DeclareMathOperator{\SO}{SO}

\DeclareMathOperator{\std}{std}
\DeclareMathOperator{\Sym}{Sym}

\DeclareMathOperator{\tr}{tr}

\newcommand{\thet}[1]{\theta^{(#1)}}
\numberwithin{equation}{section}

\newtheorem{cor}[equation]{Corollary}

\newtheorem{thm}[equation]{Theorem}
\newtheorem{conj}[equation]{Conjecture}

\theoremstyle{definition}
\newtheorem{defn}[equation]{Definition}
\newtheorem{ex}[equation]{Example}
\newtheorem{algorithm}[equation]{Algorithm}

\theoremstyle{remark}
\newtheorem{rem}[equation]{Remark}

\newenvironment{enumalph}
{\begin{enumerate}}
{\end{enumerate}}

\newcommand{\defi}[1]{\textsf{#1}} 

\renewcommand{\cong}{\simeq}

\title[Definite orthogonal modular forms]{Definite orthogonal modular forms: \\ computations, excursions, and discoveries}

\author{Eran Assaf}
\address{Department of Mathematics, Dartmouth College, 6188 Kemeny Hall, Hanover, NH 03755, USA}
\email{eran.assaf@dartmouth.edu}

\author{Dan Fretwell} 
\address{School of Mathematics, Fry Building, Woodland Road, Bristol, BS8 4ES, UK}
\email{daniel.fretwell@bristol.ac.uk}

\author{Colin Ingalls} 
\address{School of Mathematics and Statistics, 4302 Herzberg Laboratories,
  1125 Colonel By Drive, Carleton University, Ottawa, ON K1S 5B6, Canada}
\email{coliningalls@cunet.carleton.ca}

\author{Adam Logan} 
\address{The Tutte Institute for Mathematics and Computation,
  P.O. Box 9703, Terminal, Ottawa, ON K1G 3Z4, Canada; and
  School of Mathematics and Statistics, 4302 Herzberg Laboratories,
  1125 Colonel By Drive, Carleton University, Ottawa, ON K1S 5B6, Canada}
\email{adam.m.logan@gmail.com}

\author{Spencer Secord} 
\address{Department of Pure Mathematics, 5319 Mathematics and Computer Building, 200 University Ave W, University of Waterloo, Waterloo, ON N2L 3G1, Canada}
\email{spencer.e.secord@gmail.com}

\author{John Voight}
\address{Department of Mathematics, Dartmouth College, 6188 Kemeny Hall, Hanover, NH 03755, USA}
\email{jvoight@gmail.com}

\begin{document}

\begin{abstract}
We consider spaces of modular forms attached to definite orthogonal groups of low even rank and nontrivial level, equipped with Hecke operators defined by Kneser neighbours.  After reviewing algorithms to compute with these spaces, we investigate endoscopy using theta series and a theorem of Rallis.  Along the way, we exhibit many examples and pose several conjectures.  As a first application, we express counts of Kneser neighbours in terms of coefficients of classical or Siegel modular forms, complementing work of Chenevier--Lannes.  As a second application, we prove new instances of Eisenstein congruences of Ramanujan and Kurokawa--Mizumoto type.  
\end{abstract}

\maketitle

\setcounter{tocdepth}{1}
\tableofcontents

\section{Introduction}

\subsection*{Motivation and context}

The rich interplay between quadratic forms, theta series, and modular forms---together with their associated Galois representations, automorphic representations, and $L$-functions---remains a topic of broad interest in number theory.  Computational methods have developed part and parcel with theoretical advances along these lines.  This union has provided a wide range of applications, including the explicit investigation of predictions in the Langlands program.

Let $Q(x_1,\dots,x_n) \in \Z[x_1,\dots,x_n]$ be a positive definite, integral quadratic form of rank $n$ and (half-)discriminant $D$.  One may think equivalently of a lattice $\Lambda \simeq \Z^n$ embedded in $\R^n$, where the standard Euclidean norm restricts to $Q$ on $\Lambda$.  
Related to $Q$ are the forms in its \defi{genus} $\Gen Q$, the set of quadratic forms locally equivalent to $Q$ at all places.  The set of global equivalence classes in the genus define the \defi{class set} $\Cls Q$.  The class set measures the failure of the local--global principle for equivalence of forms, and by the geometry of numbers we have $\#\Cls Q < \infty$.  Complex-valued functions on the finite set $\Cls Q$ (or more generally, valued in an algebraic representation of the orthogonal group of $Q$) define a space of modular forms $M=M(\Lambda)$.  The space $M$ can be equipped with the action of Hecke operators, defined by counting classes of Kneser $p$-neighbours.  Attached to eigenforms for the Hecke action are automorphic $L$-functions.  (For more detail, see Section~\ref{sec:setup}.)  

Just as in the classification of semisimple Lie groups, significant differences among spaces of orthogonal modular forms emerge depending on the parity and size of the rank $n$.  The case of small odd rank has seen significant investigation.  For rank $n=3$ and arbitrary $D$, there is a Hecke-equivariant, functorial association to classical modular forms, first developed by Birch \cite{birch} and recently refined and generalized by Hein \cite{Hein} and Hein--Tornar{\'\i}a--Voight \cite{HeinTornariaVoight}.  For $n=5$ and (at least squarefree) (half-)discriminant $D$, Rama--Tornar{\'\i}a \cite{RamaTornaria} and Dummigan--Pacetti--Rama--Tornar{\'\i}a \cite{dprt} exhibited striking explicit connections to Siegel paramodular forms, building on previous work of Ibukiyama \cite{Ibu19}.  In both cases, the association can be understood as being furnished by Clifford algebras.

On the other hand, the situation of large rank and low level has seen recent significant strides.  Chenevier--Lannes \cite{cl} beautifully studied functoriality for orthogonal modular forms attached to even unimodular lattices of ranks $n=16,24$.  M\'egarban\'e \cite{Megarbane} also studied lattices of rank $n=23,25$ with half-discriminant $D=1$.  In both cases, the corresponding automorphic representations are unramified at all finite places.  For example, in rank $n=16$, the class set is represented by $E_{8} \oplus E_{8}$ and $E_{16}$, and the partitioning of Kneser $p$-neighbours between these classes can be expressed explicitly in terms of $\tau(p)$, the Fourier coefficients of Ramanujan's $\Delta$-function.  And for $n=24$, Chenevier--Lannes prove a congruence modulo $41$ between a Siegel modular form and a classical modular form, originally conjectured by Harder \cite{Harder}.  For a r\'esum\'e, see Examples \ref{exm:sps128} and \ref{ex:harder}.

Our initial goal in this project (which began as an undergraduate summer project of Secord) was to give formulas similar to that of Chenevier--Lannes \cite{cl} for the number of $p$-neighbours.  
However, it turned out to be necessary to change our viewpoint and consider the eigenvalues and eigenvectors of the Kneser matrices and to relate them to automorphic forms and Galois representations, as well as to investigate theta series, in order to facilitate the discovery of such formulas and to enable us to prove them. 

\subsection*{Results and contents}

With this motivation in mind, here we seek to complement the work mentioned above by considering low to moderate even rank and nontrivial discriminant $D$.  We are guided by computational discovery, and we highlight features and phenomena in this setting that we hope will be insightful in the context of the Langlands program.

After a quick setup in Section \ref{sec:setup}, we present in Section \ref{sec:algs} an implementation of algorithms for computing the Hecke module structure of definite orthogonal modular forms (at good primes), implemented in \textsf{Magma} and available online \cite{amf}.  This implementation works with an arbitrary lattice and allows arbitrary weight, and we report on its practical performance.  

We then proceed in increasing even rank $n$.  The case $n=2$ concerns genera of positive definite binary quadratic forms; the associated $L$-functions are Hecke Gr\"o{\ss}encharakters, and this can be understood already classically.  For simplicity, in the remainder of the paper we focus on trivial weight---there is already a lot to see in this case.  In Section \ref{sec:rank-four} we consider rank $n=4$.  We make explicit the transfer to Hilbert modular forms, where we have a precise understanding of the eigensystems and $L$-functions that can arise (Theorem \ref{thm: Clifford rank four}).  We exhibit by example all cases that can arise.  

Preparing to move to higher rank, in Section \ref{sub:theta-rallis} we define the theta series of an eigenform $\phi \in M(\Lambda)$, for example
\begin{equation}  \label{eqn:theta1}
\theta^{(1)}(Q)(q) = \sum_{m=0}^{\infty} r_m(Q) q^m 
\end{equation}
has coefficients $r_m(Q) \colonequals \#\{(x_1,\dots,x_n) \in \Z^n : Q(x_1,\dots,x_n)=m\}$.
The \defi{depth} of $\phi$ is the smallest $g \geq 0$ such that $\thet{g}(\phi) \neq 0$.  We then state a theorem of Rallis (Theorem \ref{RallisThm}) relating the $L$-series of an eigenform to its theta series for $g$ equal to its depth, and we consider the special cases of depth $0$ and $1$.  

In Section \ref{sec:rank-6-8}, we pursue rank $n\geq 6$.  We find many examples that we can describe completely in terms of classical modular forms.  The following statement is a simple example of what can be established this way; for this purpose, we label classical modular forms following the LMFDB \cite{lmfdb}.

\begin{thm}
Let $\Lambda$ be the lattice $A_6 \oplus A_2$ of rank $8$ and discriminant $21$.  Then $\#\Cls(\Lambda)=3$, and there are three Hecke eigenforms in $M(\Lambda)$ with eigenvalues
\[ \frac{p^7-1}{p-1} + \chi(p)p^3, \quad 
\frac{p(p^5-1)}{p-1} + a_p^2 - \chi(p)p^3, \quad
\frac{p(p^5-1)}{p-1} + b_p^2 - \chi(p)p^3
\]
for the operators $T_p$ with $p \neq 3,7$, where:
\begin{itemize}
\item $\chi$ is the quadratic character of discriminant $21$, and \item $a_{p}$ and $b_p$ are the coefficients of the classical newforms of weight $4$ and level $21$ with LMFDB labels \modform{21}{4}{c}{a} and \modform{21}{4}{c}{b}, respectively.
\end{itemize}
\end{thm}

This theorem is established in Example \ref{ex:d21}.  There, we give two other ways to write the statement in the theorem: first, in terms of the $L$-functions of these eigenforms; and second, as an explicit expression for the matrix of the Hecke operator $T_p$ acting on $M(\Lambda)$.  

Further investigations in rank $6$ led us to the following conjecture.

\begin{conj} \label{conj:wt-6-ker-theta-2-0}
  Let $G_p$ be the genus of lattices of rank $6$ and discriminant $D=p$.  
Then the kernel of $\thet{2}$ on
  $G_p$ has dimension equal to the
  number of classes in $G_p$ of lattices with no
  automorphism of determinant $-1$.  
\end{conj}

We verified this conjecture for $p < 1000$ (subject to the limitations on our ability to rigorously determine $\ker \thet{2}$); however, we do not have a heuristic or conceptual reason which might explain it.  

As the discriminant and rank increase, we soon encounter Siegel modular forms of higher genus.  In some cases these can still be related explicitly to classical modular forms via lifts---see Example \ref{ex:rank-10} for the genus of lattices containing $D_4 \oplus D_6$.  In the remaining cases, which we think of as being \emph{genuine} depth at least $2$, we consider it a feature of working with definite orthogonal modular forms that we can compute some higher genus Siegel eigenforms explicitly, but indirectly.

Finally, in Section \ref{sec:congruences} we pursue congruences  between eigenvalues of classical modular form and of nonlift Siegel eigenforms.  We propose Conjecture \ref{cong} which predicts such congruences within the framework of Eisenstein congruences.  In some cases, these congruences can easily be proven by explicit computation with orthogonal modular forms: an illustrative example is as follows.

\begin{thm} \label{thm:disc-53-again}
The congruence
\begin{equation} 
a_{1,p^2}(F) \equiv a_p(f)^2 - (1+\chi_{53}(p))p^3 + p^5+p \pmod{\frakq} 
\end{equation}
holds for all primes $p \neq 53$, where:
\begin{itemize}
\item $F \in S_4(\Gamma_0^{(2)}(53),\chi_{53})$ is a nonlift Siegel eigenform of weight $4$, level $53$, and quadratic character $\chi_{53}$ whose Hecke eigenvalues $a_{1,p}(F),a_{1,p^2}(F)$ lie in the ring of integers of the sextic number field $K\colonequals \Q(\{a_{1,p}(F),a_{1,p^2}(F)\}_p)$ defined by
\[ x^6-2x^5-290x^4-388x^3+14473x^2+11014x-81256; \]
\item $\frakq$ is the unique prime of norm $397$ in the ring of integers of $K$; and 
\item $a_p(f)$ are the Hecke eigenvalues of the classical modular form $f$ of weight $4$, level $53$, and quadratic character with LMFDB label \modform{53}{4}{b}{a}.
\end{itemize}
\end{thm}

We prove Theorem \ref{thm:disc-53-again} in Example \ref{ex:disc-53-again}.  Remarkably, we do not exhibit the Siegel eigenform directly; however, it would be interesting to do so.  

\subsection*{Acknowledgements}

The authors would like to thank Neil Dummigan for explaining how our congruences fit into the general framework of his paper \cite{db} and Gonzalo Tornar{\'\i}a for helpful discussions.  Assaf and Voight were supported by a Simons Collaboration grant (550029, to Voight). Ingalls was supported by an NSERC Discovery Grant RGPIN-2017-0462. Secord was supported by the I-CUREUS program.

\section{Setup and notation}\label{sec:setup}


In this section, we provide basic setup and notation.  For convenience and to highlight ideas, we take the ground field to be the rational numbers; however, much of what we present extends to a general totally real base field.  For further reading, see e.g.\ Greenberg--Voight \cite{gv}, Rama--Tornar{\'\i}a \cite{RamaTornaria}, or Gross \cite{Gross}.

\subsection*{Lattices}

Let $(V,Q)$ be a positive definite quadratic space over $\Q$ with associated bilinear form
$B(x,y)=Q(x+y)-Q(x)-Q(y)$ for $x,y \in V$.  
Let $\Lambda\subset V$ be a $\Z$-lattice of rank $n$ and (half-)discriminant $D$.  Rescaling $Q$, we may suppose without loss of generality that $Q(\Lambda) \subseteq \Z$, and we say that $\Lambda$ is \defi{integral}.  We say that $\Lambda$ is \defi{maximal} if $\Lambda$ is not properly contained in another integral lattice.  
%
Choosing a basis $e_1,\dots,e_n$ for $\Lambda \simeq \Z^n$, the \defi{Gram matrix} of $\Lambda$ is $(B(e_i,e_j))_{i,j=1,\dots,n} \in \M_n(\Z)$, with diagonal entries $2Q(e_i) \in 2\Z$ for $i=1,\dots,n$.  

The \defi{orthogonal group} $\OO(V)$ of $V$ is the group of $\Q$-linear automorphisms of $V$ that preserve the quadratic form, the \defi{isometries} of $V$; the orthogonal group $\OO(\Lambda)$ of $\Lambda$ is the subgroup of $\OO(V)$ that stabilizes $\Lambda$.  
If $\Lambda'=\gamma(\Lambda)$ for $\gamma \in \OO(V)$, we say $\Lambda$ is \defi{isometric} to $\Lambda'$ and we write simply $\Lambda \simeq \Lambda'$.  


Repeating these definitions but with $\Q_p$ and $\Z_p$ in place of $\Q$ and $\Z$, respectively, we can consider the completions $\Lambda_p \colonequals \Lambda \otimes \Z_p \subset V_p \colonequals V \otimes_{\Q} \Q_p$ for primes $p$.  

The \defi{genus} of $\Lambda$ is the set of lattices
\begin{equation} 
\Gen(\Lambda) \colonequals \{\Lambda' \subset V : \Lambda_p' \simeq \Lambda_p \text{ for all primes $p$}\}, 
\end{equation}
i.e., the set of lattices which become isometric to $\Lambda$ in each completion.


The orthogonal group $\OO(V)$ acts on the genus $\Gen(\Lambda)$, and we define the \defi{class set} to be the set of global isometry classes
\begin{equation} 
\Cls(\Lambda) \colonequals \OO(V)\backslash\!\genus(\Lambda).
\end{equation}
By the geometry of numbers we have $h=h(\Lambda) \colonequals \#\Cls(\Lambda) < \infty$ (see the First Finiteness Theorem of Minkowski on p.$99$ of \cite{Siegel}).  Let $\Cls(\Lambda) = \{[\Lambda_1],\dots,[\Lambda_h]\}$ with $\Lambda=\Lambda_1$. 

Following Kneser \cite{Kneser}, for a prime $p$ (allowing $p=2$) and integer $1\leq k\leq \lfloor n/2 \rfloor$, a lattice $\Pi \subset V$ is called a \defi{$p^k$-neighbour} of $\Lambda$, and we write $\Lambda \sim_{p^k} \Pi$, if there exist group isomorphisms 
\[\Lambda/(\Lambda\cap\Pi)\cong (\Z/p\Z)^k \cong \Pi/(\Lambda\cap\Pi).\]
There are evidently only finitely many $p^k$-neighbours of $\Lambda$, and if $\Pi \sim_{p^k} \Lambda$ is a $p^k$-neighbour, then $\Pi \in \Gen(\Lambda)$.  For any $p \nmid D$, the class set $\Cls(\Lambda)$ is connected under the $p$-neighbour relation, and lattices in the same genus have the same number of $p^k$-neighbours.


\subsection*{Orthogonal modular forms}

The space of \defi{orthogonal modular forms} for $\Lambda$ (of trivial weight) is the $\C$-vector space of functions on $\Cls(\Lambda)$:
\begin{equation} 
M(\Lambda) \colonequals \{\phi:\Cls(\Lambda)\rightarrow \C\}. 
\end{equation}
(We often implicitly work with the subspace of functions with values in $\Q$, or in a number field.)  A basis for this vector space is given by the characteristic functions on the set $\Cls(\Lambda)$: explicitly, we take $\phi^{(1)},\dots,\phi^{(h)}$ defined by $\phi^{(i)}([\Lambda_j])=\delta_{ij}=1,0$ according as $i=j$ or not.  For $c_1,\dots,c_h \in \C$, we abbreviate
\begin{equation}
[c_1,...,c_h] \colonequals \sum_{i=1}^h c_i\phi^{(i)} \in M(\Lambda),
\end{equation}
noting that this depends on the implicit ordering of the elements in $\Cls(\Lambda)$.

More generally, given a finite-dimensional (algebraic) representation $\rho \colon \OO(V) \circlearrowright W$, we may similarly define a space of orthogonal modular forms $M(\Lambda,W)$ of \defi{weight} $W$: these are functions on $\Cls(\Lambda)$ with values in $W$, equivariant with respect to the orthogonal group, with
\begin{equation} \label{eqn:MlambdaVWweight}
M(\Lambda,W) \simeq \bigoplus_{i=1}^h W^{\OO(\Lambda_i)}
\end{equation}
where $W^{\OO(\Lambda_i)}$ denotes the fixed subspace of $W$ under the finite group $\OO(\Lambda_i)$.  We omit the details of this case, as we mostly restrict our attention below to the case where $W$ is the trivial representation: for more detail, see e.g.~Rama--Tornar{\'\i}a \cite{RamaTornaria}*{\S 1.2} or Greenberg--Voight 
\cite{gv}*{\S 2, (4)}.

We define an inner product on $M(\Lambda)$ by
\begin{equation}  \label{eqn:innerprod}
\langle \phi^{(i)}, \phi^{(j)} \rangle = \frac{\delta_{ij}}{\#\OO(\Lambda_i)} \end{equation}
extending by linearity.  The constant function $[1,1,\dots,1] \in M(\Lambda)$ is called \defi{Eisenstein}; we define the \defi{cuspidal subspace} $S(\Lambda)\subset M(\Lambda)$ to be the orthogonal complement of the constant functions.

The $p^k$-neighbour relation defines linear operators on $M(\Lambda)$ as follows: for $p \nmid D$, we define the \defi{Hecke operator}
\begin{equation} \label{eqn:tpkhecke}
\begin{aligned}
T_{p,k} \colon M(\Lambda) &\to M(\Lambda) \\
T_p(f)([\Lambda']) &= \sum_{\Pi' \sim_{p^k} \Lambda'} f([\Pi']).
\end{aligned} 
\end{equation}
More concretely, the matrix of $T_{p,k}$ in the basis of characteristic functions has $(i,j)$-entry equal to the number of $p^k$-neighbours of $\Lambda_j$ isometric to $\Lambda_i$.  
These operators pairwise commute and are self-adjoint with respect to the inner product \eqref{eqn:innerprod} so are simultaneously diagonalizable.  The \defi{Hecke algebra} $\calH(\Lambda)$ is the finite-dimensional $\Q$-algebra generated by the Hecke operators $\{T_{p,k} : p \nmid D\}_{p,k}$; it is an Artinian commutative ring.  An \defi{eigenform} in $M(\Lambda)$ is a simultaneous eigenvector for the Hecke algebra.  

%


The Eisenstein function is always an eigenform; its eigenvalue $N_{p,k}$ under $T_{p,k}$ is the total number of $p^k$-neighbours of $\Lambda$. For example, we have 
\begin{equation}
N_{p,1} =\sum_{i=0}^{n-2} p^i +\chi_{D^*}(p)p^{\frac{n}{2}-1}
\end{equation}
where $D^* = 1$ if $n$ is odd and $D^*=(-1)^{\frac{n}{2}}D$ if $n$ is even, and $\chi_d = \left(\frac{d}{\cdot}\right)$ is the quadratic character attached to $\Q(\sqrt{d})$.  


Let $\phi \in M(\Lambda)$ be an eigenform with $T_{p,k}(\phi) = \lambda_{p,k} \phi$.  We define the (automorphic) $L$-function attached to $\phi$ as an Euler product
\[ L(\phi,s) \colonequals \prod_p L_p(\phi,p^{-s})^{-1} \]
where $L_p(\phi,T) \in 1 + T\C[T]$ is a polynomial of degree $n$ defined in terms of the eigenvalues $\lambda_{p,k}$ via the Satake transform: see Murphy \cite{Murphy}*{\S 3} for an explicit description and precise formulas in ranks $n \leq 8$ \cite{Murphy}*{p.~56--57}.  For example, for $n=4$ we have
\begin{equation} 
L_p(\phi,T) = 
\begin{cases}
1 - \lambda_{p,1}T + p(\lambda_{p,2}+2)T^2 - \lambda_{p,1}p^2 T^3 + p^4 T^4, & \text{ if $\chi_{D^*}(p)=1$;} \\
(1-pT)(1+pT)(1 - \lambda_{p,1}T + p^2T^2), & \text{ if $\chi_{D^*}(p)=-1$}.
\end{cases}
\end{equation} 
In particular, note that in both cases the coefficient of $T$ is $-\lambda_{p,1}$. This is the case for arbitrary rank, a fact that we will need to use later.

\section{Algorithms} \label{sec:algs}

In this section, we review algorithms for computing the Hecke module structure of orthogonal modular forms, and we report on our implementation in \textsf{Magma} \cite{Magma}, available online \cite{amf}.  This implementation allows a general totally real base field $F$, but again for simplicity we restrict our presentation to the case $F=\Q$.  For further background reading, see e.g.\ Greenberg--Voight \cite{gv}*{Section 6}.

\subsection*{Algorithms}


The algorithms we require include the following:
\begin{enumerate}
\item \label{item: omf} \texttt{OrthogonalModularForms}$(\Lat, W)$: construct from a lattice $\Lat$ and a weight $W$ a basis for the space of orthogonal modular forms $M(\Lat, W)$. The returned data type stores the genus of the lattice $\Gen(\Lat)$ and the bases for each subspace $W^{\OO(\Lat_i)}$. 
\item \label{item: Hecke operator} \texttt{HeckeOperator}$(M, p, k)$: the matrix representing the Hecke operator $T_{p,k}$ on the space $M = M(\Lat, W)$ (with respect to the computed basis).
\item \label{item: Hecke Eigenforms} \texttt{HeckeEigenforms}$(M)$: a list of eigenforms for the Hecke algebra, with one representative for every Galois orbit. 
\item \label{item: Hecke Eigenvalue} \texttt{HeckeEigenvalue}$(f, p, k)$: for an eigenform $\phi$, the eigenvalue $\lambda_{p,k}$ such that $T_{p,k}(\phi) = \lambda_{p,k} \phi$. 
\item \label{item: LPolynomial} \texttt{LPolynomial}$(f,p)$: the $L$-polynomial $L_p(\phi,T)$ of the eigenform $\phi$.
\end{enumerate}

For \texttt{OrthogonalModularForms}, in view of \eqref{eqn:MlambdaVWweight} we need to enumerate the genus, and then compute automorphism groups of the lattices; we obtain a basis by computing fixed subspaces via standard linear algebra.  The enumeration of representatives of the genus of a lattice using $p$-neighbours has been studied in great detail, with many practical improvements.

We briefly elaborate upon the main workhorse \texttt{HeckeOperator} \eqref{item: Hecke operator}.  By \eqref{eqn:tpkhecke}, the Hecke operators are obtained by summing over $p^k$-neighbours.  An algorithm for computing Hecke operators using $p^k$-neighbours is described in generality (allowing for other algebraic groups and arbitrary weights) in Greenberg--Voight \cite{gv}; it was implemented for orthogonal modular forms of trivial weight in Magma \cite{Magma} by Greenberg, Jeffery Hein, and Voight.  (For lattices over number fields, we rely upon an implementation of Markus Kirschmer and David Lorch.)  Beyond enumerating $p^k$-neighbours using isotropic subspaces, it relies on the algorithm of Plesken--Souvignier \cite{PS} for isometry testing between lattices, which was implemented in \textsf{Magma} \cite{Magma} by Souvignier, with further refinements to the code contributed by Allan Steel, Gabriele Nebe, and others. 

\begin{algorithm}[\texttt{HeckeOperator}$(M,p,k)$] \label{alg: HeckeOperator}
\ 
\begin{algorithmic}[1]
\STATE Let $\gen(\Lat) = \{ \Lat_1, \ldots, \Lat_h \}$ be the genus representatives.
\STATE Let $\{ v_{(i,l)} \}_{i,l}$ be a basis for $M$ such that $\{ v_{(i,l)} \}_l$ is a basis for $W^{O(\Lat_i)}$. 
\FOR{$i = 1,2,\ldots, h$}
\STATE Let $t_{i,(j,m)} \colonequals 0$ for all $m$ and $j$. 
\FOR{$\Lat' \sim_{p^k} \Lat_i$}
\STATE Find $j$ and $\gamma_{ij} \in \OO(V)$ such that $\Lat' = \gamma_{ij} \Lat_j$ by isometry testing.
\STATE Let $t_{i,(j,m)} +\colonequals \gamma_{ij} v_{(j,m)}$ for all $m$.
\STATE Write $t_{i,(j,m)} = \sum_{l} t_{(i,l),(j,m)} v_{(i,l)} $ for all $m$ and $j$.
\ENDFOR
\ENDFOR
\RETURN $T = (t_{(i,l),(j,m)})$.
\end{algorithmic}
\end{algorithm}

The complexity of Algorithm \ref{alg: HeckeOperator} is dominated by $O(h^2 p^{k(n-k-1)})$ isometry tests between lattices, if done \emph{naively}---for a refined approach, see more on isometry testing below.  
Building on their implementation, Assaf extended the implementation to support higher rank lattices, Hermitian lattices (for unitary groups), and arbitary weight.  This implementation includes highest weight representations for orthogonal and unitary groups in characteristic $0$ as explained e.g.~by Fulton--Harris \cite{FultonHarris}.  (An implementation of these representations over finite fields exists in \textsf{Magma} \cites{CMT,deGraaf} by Willem de Graaf and others, based on the LiE system \cite{vanLeeuwen}.)  

\texttt{HeckeEigenforms}  \eqref{item: Hecke Eigenforms} is accomplished using linear algebra on the output of sufficiently many calls to \texttt{HeckeOperator}.  We compute the operators $T_{p} = T_{p,1}$ for small primes until the simultaneous eigenspaces are irreducible over $\Q$.  For large enough spaces with a maximal lattice, the single operator $T_2$ often suffices in practice to observe eigensystems occuring with multiplicity one.  

\begin{rem}
We expect that this multiplicity one phenomenon can be explained by work of Aizenbud--Gourevitch--Rallis--Schiffmann \cite{agrs}.  However, when the lattice is not maximal, multiplicity one need not hold due to the presence of oldforms.  It would be interesting to study these multiplicities in more detail.
\end{rem} 

\texttt{HeckeEigenvalue} \eqref{item: Hecke Eigenvalue} slightly improves on the preceding by using the fact that $\phi$ is an eigenform.  Indeed, if we write 
$\phi = \sum_i \phi^{(i)}$ where $\phi^{(i)} \in W^{\OO(\Lambda_i)}$ (overlapping with previous notation), we may choose an index $i$ such that $\phi^{(i)} \ne 0$ and compute only loop number $i$ (Step 3) in Algorithm \ref{alg: HeckeOperator}. This already yields $T_{p,k} \phi^{(i)} = \lambda_{p,k} \phi^{(i)}$, from which we can extract $\lambda_{p,k}$. Thus, Algorithm \texttt{HeckeEigenvalue} \eqref{item: Hecke Eigenvalue} saves a factor $h$ in its running time in comparison to \texttt{HeckeOperator}.  

Finally, \texttt{LPolynomial} \eqref{item: LPolynomial} first uses \texttt{HeckeEigenvalue} \eqref{item: Hecke Eigenvalue} to compute the eigenvalues $\lambda_{p,k}$ for $k=1,2,\ldots,\lfloor n/2 \rfloor$.  It then produces the $L$-polynomial from these eigenvalues using the Satake transform, as described by Murphy \cite{Murphy}*{\S 3}.  The running time complexity of Algorithm \eqref{item: LPolynomial} is dominated by $O(hp^{n(n-2)/4})$ applications of isometry testing.

The running time is polynomial in $p$ (exponential in $\log p$), and the exponent is quadratic in the rank $n$, making computations in very high rank almost infeasible.  However, in some of the applications described in the paper, we only require knowledge of the $L$-polynomial at a single prime. In any case, even improvements by constant factors (depending on the rank $n$) are of practical importance.
We turn now to discuss several such improvements. 

\subsection*{Genus enumeration}
\label{rem:unbalanced-lattices-copy}
For some of the genera appearing in our examples, a straightforward attempt to find all of the lattices in the genus and their automorphism groups using \textsf{Magma} takes a long time.  As an example, consider the genus of lattices of rank $8$ and discriminant $p \equiv 1 \pmod{4}$. One of the lattices in this genus is generated by $E_7$ and a vector of norm $(p+3)/2$, and \defi{Magma}'s algorithm for finding the automorphism group of a lattice relies on listing all of the vectors of norm up to $m$, where $m$ is minimal such that these vectors span a sublattice of rank $n$ (equivalently, of finite index)---hence unnecessarily enumerating all elements of $E_7$ of norm up to $(p+3)/2$.  

Once this problem is recognized, it is easily dealt with: we compute directly with this lattice, relating its automorphism group to that of $E_7$.  In cases where $p \equiv 1 \pmod 8$ or $p \equiv 1 \pmod{12}$, there are lattices generated by $A_7$, $D_7$, or $E_6 \oplus A_1$ and one vector of large norm that cause similar (but less severe) problems.  More generally, if we were trying to enumerate genera of lattices of rank $n$ we would directly find the lattices that have a large root sublattice of small discriminant and their automorphism groups. 

In light of this issue, our implementation offers an option for the user to supply the lattices in a genus together with their automorphism groups.

\subsection*{Isometry testing}

To test for isometry, we rely on standard algorithms for $\Z$-lattices. Since our genus representatives are fixed while computing Hecke operators, we are able to perform some precomputation steps in order to improve the running times. If $\Lat_1,\dots,\Lat_h$ are representatives for $\Cls(\Lambda)$, we compute the first few coefficients of its theta series $\theta^{(1)}(\Lambda_i)$ (as defined in \eqref{eqn:theta1}) and cache them before enumerating the $p$-neighbours.  Since these are isometry invariants, we can compute them for every $p$-neighbour, and test for isometry only when they match.  If the cached data determines the genus representative uniquely and the weight is trivial, we do not need to test for isometry at all, eliminating the need for any isometry testing.  In higher weight, one needs to compute the actual isometry, but this can be computed on the correct representative so the total number of isometry tests is equal to the number of neighbours.

There are several other possible ways to exploit the ability to precompute data in order to reduce the running time of isometry testing.  For example, the ultimate representative would be to compute a \emph{canonical form} for the lattice \cite{CanonicalForm} (or in rank $\leq 4$, we could compute a Minkowski-reduced representative).  Although this seems to work very well when tested on its own terms, we have not been able to take advantage of this speedup in computing modular forms because its implementation does not easily plug into our implementation in \textsf{Magma}.  We have also attempted to use greedy reduction, as described by Nguyen--Stehl\'e \cite{NS}. However, as the reduction process does not yield a unique representative, one has to determine the orbits of the greedy-reduced lattices. The precomputation of these orbits turned out to be slower than computing the Hecke operators $T_{p,k}$ in practice. 

\subsection*{Automorphism group and time/memory trade-off}

The algorithms \texttt{HeckeOperator} and \texttt{HeckeEigenvalue} for computing the Hecke operator $T_{p,k}$ and its eigenvalues has naive running time complexity of $O(h^2 p^{k(n-k-1)})$ isometry tests, while requiring only $O(1)$ memory.  In the presence of memory resources, we leverage this to gain some improvement, even if by a constant, as follows.

The group $\OO(\Lat)$ acts on the set of $p^k$-neighbours by isometries, hence it suffices to test isometries on a set of orbit representatives. The naive time/memory trade-off is then to precompute the orbits of $\OO(\Lat)$ on neighbours by union find, at the cost of $O(p^{k(n-k-1)})$ memory. An alternative is obtained by keeping only a single orbit in memory at any given time, expanding it while computing its stabilizer.  In both cases, if we are computing \texttt{HeckeEigenvalue}, we can choose an index $i$ such that $\#\OO(\Lat_i)$ is maximized.

\subsection*{Timings}

We record the performance of our implementation. All the timings appearing here were measured on a standard desktop machine.
Each example has a corresponding code snippet included in the examples in our package \cite{amf}.   

\begin{ex} \label{ex:rank-4-disc-1369}
We consider the genus of maximal, integral lattices of rank $4$ and discriminant $D=37^2=1369$.  We compute $L$-polynomials for the eigenforms.  
A representative $\Lambda$ of the genus 
corresponds to the quadratic form
\begin{equation} 
Q(x,y,z,w) = x^2 + xz+xw + 2y^2 + yz + 2yw + 5z^2 + zw + 10w^2. 
\end{equation}
Running \texttt{OrthogonalModularForms}, we find $\#\Cls(\Lambda)=4$, with
representatives 
\begin{equation}
\begin{aligned}
\Lat_1 = \Lat = 
\left(
\begin{array}{cccc}
2 & 0 & 1 & 1 \\
0 & 4 & 1 & 2 \\
1 & 1 & 10 & 1 \\
1 & 2 & 1 & 20
\end{array}
\right), \quad &
\Lat_2 = 
\left(
\begin{array}{cccc}
2 & 1 & 0 & -1 \\
1 & 8 & -1 & -4 \\
0 & -1 & 10 & -2 \\
-1 & -4 & -2 & 12
\end{array}
\right), \\
\Lat_3 = 
\left(
\begin{array}{cccc}
4 & -1 & -1 & 0 \\
-1 & 4 & 2 & -1 \\
-1 & 2 & 6 & 2 \\
0 & -1 & 2 & 20
\end{array}
\right), \quad &
\Lat_4 = 
\left(
\begin{array}{cccc}
4 & -1 & -1 & 1 \\
-1 & 6 & 3 & -1 \\
-1 & 3 & 8 & 1 \\
1 & -1 & 1 & 10
\end{array}
\right).
\end{aligned}
\end{equation}
The first three Hecke operators have matrices (under the standard basis): \[T_{2,1} = \left(\begin{array}{cccc}1&1&1&1\\ 2&4&0&2\\ 2&0&4&2\\ 4&4&4&4\end{array}\right), \quad T_{3,1} = \left(\begin{array}{cccc}4&1&1&2\\ 2&9&0&3\\ 2&0&9&3\\ 8&6&6&8\end{array}\right), \quad T_{5,1} = \left(\begin{array}{cccc}4&4&4&4\\ 8&10&6&8\\ 8&6&10&8\\ 16&16&16&16\end{array}\right)\] The corresponding eigenforms and eigenvalues are: \begin{align*}\phi_1 &= [1,1,1,1] & \lambda_{2,1} &= 9, & \lambda_{3,1} &= 16, & \lambda_{5,1} &= 36, &\dots\\ 
\phi_2 &= [0,1,-1,0] & \lambda_{2,1} &= 4, & \lambda_{3,1} &= 9, & \lambda_{5,1} &= 4, &\dots\\ 
\phi_3 &= [4,-2,-2,1]  & \lambda_{2,1} &= 0, & \lambda_{3,1} &= 4, & \lambda_{5,1} &= 0, &\dots
\\ \phi_4 &= [4,1,1,-2] & \lambda_{2,1} &= 0, & \lambda_{3,1} &= 1, & \lambda_{5,1} &= 0, &\dots\end{align*} 

Below are the timings (in seconds) measured to produce $L$-polynomials for $p < 100$.  

\begin{table}[ht!]
\centering
\begin{tabular}{ |c|c|c|c|c|c|c|c|c|c|c|c|c|c| } 
 \hline
 $p$ & 2 & 3 & 5 & 7 & 11 & 13 & 17 & 19 & 23 & 29 & 31 & 37 & 41  \\ 
 \hline
 $\phi_1$ & 0.00& 0.01& 0.04& 0.07& 0.13& 0.16& 0.27& 0.33& 0.45& 0.71& 0.84& 0.00& 1.40 \\ 
 $\phi_2$ & 0.01& 0.02& 0.06& 0.10& 0.20& 0.26& 0.42& 0.55& 0.80& 1.22& 1.37& 0.01& 2.41 \\ 
 \hline
\end{tabular}
\begin{tabular}{ |c|c|c|c|c|c|c|c|c|c|c|c|c| } 
 \hline
 $p$ & 43 & 47 & 53 & 59 & 61 & 67 & 71 & 73 & 79 & 83 & 89 & 97 \\ 
 \hline
 $\phi_1$ & 1.50& 1.79& 2.24& 2.86& 3.04& 3.57& 4.01& 4.19& 4.99& 5.45& 6.28& 7.38  \\ 
 $\phi_2$ & 2.62& 3.12& 4.03& 5.09& 5.40& 6.54& 7.37& 7.44& 8.87& 9.80& 11.25& 13.46 \\ 
 \hline
\end{tabular}
\begin{adjustwidth}{2cm}{0cm}
\captionsetup{singlelinecheck=off, skip=4pt, width =\dimexpr \textwidth-2cm\relax}
\caption{Timings for a lattice of rank $4$ and $D = 37^2$}
\end{adjustwidth}
\label{table: rank 4 disc 37_2}
\end{table}

\noindent Note that running times for $\phi_2$ are longer. This is due to the fact that the support of $\phi_1$ (and $\phi_3,\phi_4$) includes a lattice with $\#\OO(\Lat) = 8$, while the support of $\phi_2$ only includes lattices with $\#\OO(\Lat) = 4$.  
Note also that $p=37$ is significantly faster, which is due to the ramification at $37$. 
\end{ex}



\begin{ex}\label{ex:rank-4-disc-193}
We consider the genus of maximal integral lattices of rank $4$ and $D = 193$. 
We find that $\#\Cls(\Lat) = 9$.
Below are the timings (in seconds) measured to produce $L$-polynomials for $p < 100$. 
\begin{table}[ht!]
\centering
\begin{tabular}{ |c|c|c|c|c|c|c|c|c|c|c|c|c|c| } 
 \hline
 $p$ & 2 & 3 & 5 & 7 & 11 & 13 & 17 & 19 & 23 & 29 & 31 & 37 & 41  \\ 
 \hline
 $\phi$ &  0.01& 0.03& 0.02& 0.06& 0.05& 0.07& 0.12& 0.15& 0.34& 0.32& 0.61& 0.47& 0.58 \\ 
 \hline
\end{tabular}
\begin{tabular}{ |c|c|c|c|c|c|c|c|c|c|c|c|c| } 
 \hline
 $p$ & 43 & 47 & 53 & 59 & 61 & 67 & 71 & 73 & 79 & 83 & 89 & 97 \\ 
 \hline
 $\phi$ & 1.10& 0.78& 1.00& 1.97& 1.36& 2.58& 1.80& 1.89& 2.21& 3.88& 2.81& 5.42  \\ 
 \hline
\end{tabular}
\caption{Timings for a lattice of rank $4$ and $D = 193$}
\label{table: rank 4 disc 193}
\end{table}
Note that in this case all forms have support including the lattice with the largest automorphism group.  The time taken is closely approximated by  $c_{\chi_{193}(p)} p^2$ seconds, where $c_1/c_{-1}$ is roughly $1.62$; this is due to the fact that for inert primes, there are fewer neighbours.
\end{ex}

\begin{ex}\label{ex:rank-6-disc-39}
We consider a genus of (maximal) integral lattices of rank $6$ and $D = 39$ containing a lattice $\Lambda \cong A_2 \oplus \Lambda_2$, where $\Lambda_2$ is a lattice of rank $4$ generated by $A_3$ and a vector of norm $4$ whose intersections with the $3$ roots corresponding to the vertices of the Dynkin diagram are $1, 0, 0$.  It takes less than a second to set up the space (of dimension $2$) and compute the two eigenforms.
We give timings (in seconds) measured to produce $L$-polynomials for $p < 20$ in Table \ref{table: rank 6 disc 39}.
\begin{table}[ht!]
\centering
\begin{tabular}{ |c|c|c|c|c|c|c|c|c| } 
 \hline
 $p$ & 2 & 3 & 5 & 7 & 11 & 13 & 17 & 19  \\ 
 \hline
 $\phi$ &  0.13& 0.30& 4.11& 12.05& 185.39& 341.80& 1209.94& 2228.58 \\ 
 \hline
\end{tabular}
\caption{Timings for a lattice of rank $6$ and $D = 39$}
\label{table: rank 6 disc 39}
\end{table}
\end{ex}

\begin{ex}\label{ex:rank-6-disc-75}
Consider the genus of integral lattices of discriminant $D = 75$ that contains the lattice $A_4 \oplus \Lambda_{15}$, where $\Lambda_{15}$ is a lattice of rank $2$ spanned by vectors $x, y$ of norm $4$ with $(x,y) = 1$.
We give timings (in seconds) measured to produce $L$-polynomials for $p < 20$ in Table \ref{table: rank 6 disc 75}.
\begin{table}[ht!]
\centering
\begin{tabular}{ |c|c|c|c|c|c|c|c|c| } 
 \hline
 $p$ & 2 & 3 & 5 & 7 & 11 & 13 & 17 & 19  \\ 
 \hline
 $\phi$ & 0.03& 0.23& 3.97& 14.31& 126.77& 358.64& 1055.34& 2256.81 \\ 
 \hline
\end{tabular}

\caption{Timings for a lattice of rank $6$ and $D = 75$}
\label{table: rank 6 disc 75}
\end{table}
\end{ex}

\begin{ex}\label{ex:rank-6-disc-84}
Consider the genus of integral lattices with discriminant $D = 84$ that contains $\Lat = A_1^2 \oplus A_2 \oplus L_7$, where $L_7$ is the lattice of rank $2$ and discriminant $7$.
We give timings (in seconds) measured to produce $L$-polynomials for $p < 20$ in Table \ref{table: rank 6 disc 84}.
\begin{table}[ht!]
\centering
\begin{tabular}{ |c|c|c|c|c|c|c|c|c| } 
 \hline
 $p$ & 2 & 3 & 5 & 7 & 11 & 13 & 17 & 19  \\ 
 \hline
 $\phi$ & 0.05 & 0.27 & 3.59 & 12.58 & 187.29 & 358.52 & 1489.43 & 2604.32 \\ 
 \hline
\end{tabular}

\caption{Timings for a lattice of rank $6$ and $D = 84$}
\label{table: rank 6 disc 84}
\end{table}
\end{ex}

\begin{ex}\label{ex:rank-6-disc-131}
Consider the genus of integral lattices of rank $6$ with discriminant $D = 131$.
We give timings (in seconds) measured to produce $L$-polynomials for $p < 20$ in Table \ref{table: rank 6 disc 131}.
\begin{table}[ht!]
\centering
\begin{tabular}{ |c|c|c|c|c|c|c|c|c| } 
 \hline
 $p$ & 2 & 3 & 5 & 7 & 11 & 13 & 17 & 19  \\ 
 \hline
 $\phi$ & 0.11 & 0.54 & 4.08 & 18.52 & 202.14 & 488.63 & 1323.29 & 2284.46 \\ 
 \hline
\end{tabular}

\caption{Timings for a lattice of rank $6$ and $D = 131$}
\label{table: rank 6 disc 131}
\end{table}
\end{ex}

\begin{ex}\label{ex:rank-8-disc-21}
We consider the genus of lattices of rank $8$ and $D=21$ containing $\Lambda = A_6 \oplus A_2$.  It takes 14 seconds to compute the space (dimension $3$) and $0.23$ seconds to compute eigenforms,
and the following much shorter Table \ref{table: rank 8 disc 21} shows how long it takes to compute $L$-polynomials.
\begin{table}[ht!]
\centering
\begin{tabular}{ |c|c|c|c| } 
 \hline
 $p$ & 2 & 3 & 5 \\ 
 \hline
 $\phi$ & 1.25 & 62.8 & 93955.09 \\ 
 \hline
\end{tabular}
\caption{Timings for a lattice of rank $8$ and $D = 21$}
\label{table: rank 8 disc 21}
\end{table}
\end{ex}

\begin{ex}\label{ex:rank-8-disc-53}
We consider the unique genus of lattices of rank $8$ and $D=53$.  It takes 305 seconds to compute the space (dimension $8$) and $1$ second to compute eigenforms,
and the following Table \ref{table: rank 8 disc 53} shows how long it takes to compute $L$-polynomials.
\begin{table}[ht!]
\centering
\begin{tabular}{ |c|c|c|c|c| } 
 \hline
 $p$ & 2 & 3 & 5 & 7 \\ 
 \hline
 $\phi$ & 2.43 & 71.62 & 9559.27 & 345324.19 \\ 
 \hline
\end{tabular}
\caption{Timings for a lattice of rank $8$ and $D = 53$}
\label{table: rank 8 disc 53}
\end{table}
\end{ex}

\begin{ex} \label{ex:rank-6-disc-39-weight-2}
We consider the lattice from Example~\ref{ex:rank-6-disc-39}, but consider forms of weight $(2,0,0)$.
It takes us less than a second to find that the dimension of the space is $4$, and that it consists of two Galois orbits of eigenforms, of sizes $1,3$.
The following table shows how long it takes to compute $L$-polynomials for either eigenform:
\begin{table}[ht!]
\centering
\begin{tabular}{ |c|c|c|c|c| } 
 \hline
 $p$ & 2 & 3 & 5 & 7 \\ 
 \hline
 $\phi$ & 0.66 & 0.98  & 19.81 & 37.7 \\ 
 \hline
\end{tabular}
\label{table: rank 6 disc 39 weight (2,0,0)}
\caption{Timings for a lattice of rank $6$, $D = 39$ and weight $(2,0,0)$}
\end{table}
\end{ex}


\begin{ex} \label{ex:rank-6-disc-7}
We consider the root lattice $A_6$, of rank $6$ and discriminant $7$. In trivial weight it only admits an Eisenstein series, but in weight $(4,0,0)$ we find a cusp form $\phi$ in 10 seconds.
The following table shows how long it takes to compute $L$-polynomials.
\begin{table}[ht!]
\centering
\begin{tabular}{ |c|c|c|c|c| } 
 \hline
 $p$ & 2 & 3 & 5 & 7 \\ 
 \hline
 $\phi$ & 181.56 & 468.45 & 4632.85 & 10253.26 \\ 
 \hline
\end{tabular}
\label{table: rank 6 disc 7 weight (4,0,0)}
\caption{Timings for a lattice of rank $6$, $D = 7$ and weight $(4,0,0)$}
\end{table}
\end{ex}



\begin{ex} \label{ex:A1011}
We consider the root lattice $A_{10}$ of rank $10$ and discriminant $11$. 
We find that the genus consists of $3$ lattices, giving $3$ distinct eigenforms, $\phi_1, \phi_2, \phi_3$. We can compute the polynomials $L_2(\phi_i,T)$ for $i=1,2,3$ in $249.52$ seconds.
\end{ex}


\begin{ex} \label{ex:rank-10-disc-27}
We consider the genus of lattices of rank 10 and discriminant 27 that contains $E_6 \oplus A_2^2$.
We find that the genus consists of $2$ lattices, giving a single cusp form, $\phi$. We can compute the polynomial $L_2(\phi,T)$ in $264.51$ seconds.
\end{ex}

\section{Rank four}\label{sec:rank-four}

In this section, we consider spaces $M(\Lat)$ where $\Lat$ has rank $n = 4$.  In this case, we relate orthogonal modular eigenforms explicitly to Hilbert modular forms, and we give examples.



\subsection*{Transfer}

Let $\Lambda$ be a lattice of rank $4$, as in section \ref{sec:setup}.  In this section, we suppose that $\Lambda$ is maximal, to simplify the discussion of newforms and oldforms.  Write its discriminant as $D = D_0 N^2$ where $D_0$ is a fundamental discriminant. The orthogonal modular forms for $\Lambda$ will be described as Hilbert modular forms over the \'etale algebra 
\begin{equation}
K \colonequals \Q[\sqrt{D_0}] = \Q[x]/(x^2-D_0)
\end{equation}
So if $D_0=1$ we have $K \simeq \Q \times \Q$ and will again find classical modular forms, otherwise we have a real quadratic field. Let $\Z_K$ be the ring of integers of $K$, with $\Z_K=\Z \times \Z$ if $D_0=1$.

To further focus on a clarifying case, we explain the precise relation in the case where $N$ is squarefree.  We say that a prime $p$ is \defi{isotropic} for $V$ if there exists nonzero $x \in V \otimes \Q_p$ such that $Q(x) = 0$; else, we say that $p$ is anisotropic.  There are finitely many anisotropic primes, and we let $M$ be their product. 


For an integer $N$, write $S_2(N \Z_K)$ for 
the space of Hilbert cusp forms of parallel weight $2$, level $N \Z_K$, and trivial character.  This finite-dimensional $\C$-vector space comes equipped with a Hecke algebra $\Hecke(N \Z_K)$ of operators away from $N$ as well as a cavalcade of additional structures, as follows.
\begin{enumerate}
\item The space $S_2(N\Z_K)$ decomposes into new and old subspaces; we let $S_2(N\Z_K)^{M\textup{-new}}$ be the space of forms which are new at all primes $\frakp \mid M$.  
\item For every $p \mid N$, there exists an involution $W_p$ on this space, called the \defi{Atkin--Lehner involution} at $p$.  (When $p$ splits, this is the product of the involutions for the two primes above $p$.)
For a sequence $\{c_{p}\}_{p \mid N}$ with $c_p \in \{ \pm 1\}$, we write $S_2(N \Z_K; \{ c_{p} \}_{p \mid N})$ for the subspace of forms $f \in S_2(N \Z_K)$ such that $W_p f = c_p f$. 
\item The Galois group $G_K \colonequals \Gal(K\,|\,\Q) = \langle \sigma \rangle$ acts on $S_2(N\Z_K)$ via its action on the base field: in terms of Hecke eigenvalues, we have $a_{\frakp}(\sigma f) = a_{\sigma(\frakp)}(f)$.  
\item There is a twisting action by the group of finite order Hecke characters of modulus $N\Z_K (\infty)$.  We denote by $X=X(N)$ the Hecke characters that act on $S_2(N \Z_K; \{ c_p \}_p)$: 
\begin{equation}
X \colonequals \{ \chi : \Cl^{+}(N \Z_K) \to \C^{\times} : \chi^2 = 1, \ \text{$\chi(p) = 1$ for all $p \mid N$} \}.
\end{equation}
\end{enumerate}

Putting these altogether, we write 
\begin{equation}
G_K \backslash S_2(N \Z_K; \{ c_{p} \}_p)^{X, M \text{-new}}
\end{equation}
for the subspace of forms which are $M$-new at all primes $\frakp \mid M$ and fixed by all characters in $X$, up to the swapping action of $G_K$.  

The following transfer of modular forms can be proven using the even Clifford functor. 

\begin{thm}[\cite{future2}] \label{thm: Clifford rank four}
There is an injective linear map from orthogonal cusp forms to orbits of Hilbert cusp forms
\[
C_0 \colon S(\Lat) \hookrightarrow G_K \backslash S_2(N \Z_K)
\]
and a natural embedding $\Hecke(\Lat) \hookrightarrow \Hecke(N \Z_K)$ for which this injection is equivariant for the action of the corresponding Hecke algebras.  The image of this map consists of the orbits in $S_2(N\mathbb{Z}_K;\{c_p\}_{p\mid N})^{X, M\text{-new}}$, where $c_p = -1,1$ according as $p \mid M$ or not. 
\end{thm}

In the square discriminant case, i.e., $D_0 = 1$, this was proved by B{\"o}cherer and Schulze-Pillot in \cite{BSP}. An equality of dimensions can be deduced from the results of Ponomarev in \cite{Ponomarev}. 

As a corollary from this description, we obtain a relationship between the $L$-polynomials of the objects on both sides. In order to describe this relation we recall the definition of the Asai $L$-function associated to a Hilbert eigenform. Let $f \in S_2(N \Z_K)$ be an eigenform. For every prime $\frakp$ of $\Z_K$ that does not divide $N$ with $T_{\frakp} f = a_{\frakp} f$, we factor 
\begin{equation}
1 - a_{\frakp} T + \Nm(\frakp) T^2 = (1 - \alpha_{\frakp}T)(1 - \beta_{\frakp}T)
\end{equation}
where $\Nm(\frakp)$ is the absolute norm. Asai \cite{Asai} defines for every prime $p \nmid N$ a polynomial depending on the splitting behavior of $p$ in $K$:
\begin{equation} \label{eq: Asai L function}
L_p(f, T, \Asai) \colonequals \begin{cases} 
(1 - \alpha_{\frakp} \alpha_{\frakp'}T)
(1 - \alpha_{\frakp} \beta_{\frakp'}T)
(1 - \beta_{\frakp} \alpha_{\frakp'}T)
(1 - \beta_{\frakp} \beta_{\frakp'}T), & \text{ if $p \Z_K = \frakp \frakp'$}; \\
(1 - \alpha_{\frakp} T) (1 - \beta_{\frakp} T) (1 - p^2 T^2), & \text{ if $p \Z_K = \frakp$}.
\end{cases}
\end{equation}
These are the ``good" $L$-polynomials of the Asai lift of $f$ to $\text{GL}_4$.
The precise description of the embedding of Hecke algebras $\Hecke(\Lat) \hookrightarrow \Hecke(N \Z_K)$ in Theorem~\ref{thm: Clifford rank four} then yields the following corollary.
\begin{cor} \label{cor: Asai L function}
Let $\phi \in S(\Lat)$ be an eigenform. Then for every prime $p \nmid D$ we have
\begin{equation}
L_p(\phi, T) = L_p( C_0(\phi), T, \Asai).   
\end{equation}
\end{cor}

\begin{rem}\label{rmk sq}
In the square discriminant case, i.e. $D_0=1$, we have $K \simeq F \times F$ and $S_2(N \Z_K) \subseteq M_2(N) \otimes M_2(N)$ is the subspace spanned by pairs of (classical) modular forms of level $N$ such that either both are cusp forms, or one form is a cusp form and the other is the Eisenstein series $E_2$. This space was named the ``essential" subspace in \cite{BSP}. In this case, all the primes are split, and so the Asai $L$-function in \eqref{eq: Asai L function} turns out to be simply the Rankin--Selberg $L$-function associated to $f \otimes g \in S_2(N \Z_K)$. Namely, if $C_0(\phi) = f \otimes g$, then $L_p(\phi, T) = L_p(f \otimes g, T)$. 
We further note that the Galois action here is simply the swap, identifying $f \otimes g$ with $g \otimes f$.
\end{rem}

\begin{rem} \label{rmk:allspinor}
In view of the description of the image of the map $C_0$ in Theorem~\ref{thm: Clifford rank four}, one might wonder where all the other forms went.  Indeed, working with a compact form we only expect to see Hecke characters whose associated Dirichlet character is trivial, and since the Clifford functor is trivial on scalars, we must also restrict to forms with trivial Hecke character. However, it is possible to obtain the spaces of forms with different Atkin--Lehner eigenvalues by using appropriate weights.
For $d \mid N$, we let $\nu_d : \Q^{\times}_{>0} / \Q^{\times 2} \to \{\pm 1 \}$ be the character defined on primes by $\nu_d(p) = -1$ iff $p \mid d$. 
Let $\gamma_0 \in O(V)$ be an isometry with determinant $-1$. Elements in $\SO(V)$ can be represented as composition of reflections by vectors, and the product of the norms of these vectors is invariant up to squares, yielding a map called the spinor norm, $\nrd : \SO(V) \to \Q^{\times}_{>0} / \Q^{\times 2}$, which we extend to $\OO(V)$ by setting $\nrd(\gamma_0) = 1$. 
Then $\psi_d = \nu_d \circ \nrd$ is a character of $\OO(V)$, known as the \defi{spinor norm} character, as in Hein--Tornar{\'\i}a--Voight \cite{HeinTornariaVoight}. One can find the forms with other Atkin--Lehner eigenvalues by considering the space of orthogonal modular forms with weight given by the spinor norm character. Full details and more general statements will be given in future work \cite{future2}. 
\end{rem}


\subsection*{Square discriminant case}

We now proceed to give examples that exhaust all possible types of eigensystem and $L$-function in the rank $4$ case.  Throughout $p$ is assumed to be a good prime, i.e. $p\nmid D$.  We begin with the case where $D_0=1$.  By Remark \ref{rmk sq} the eigenforms $\phi \in M(\Lat)$ can only belong to one of four types.

\begin{ex}
Let $\Lat$ be a maximal integral lattice with $D = 37^2$, as in Example \ref{ex:rank-4-disc-1369}. 
As expected, the eigenform $\phi_1$ is Eisenstein, having eigenvalue \[\lambda_{p,1} = (1+p+p^2) + p = (1+p)^2\] and $L$-polynomials \[L_p(\phi_1,T) = (1-T)(1-pT)^2(1-p^2T)\] (the $L$-polynomial of the ``Asai $L$-function" of $E_2\otimes E_2$).

To explain the eigenforms $\phi_2$ and $\phi_3$ we let
\begin{equation}
\begin{aligned}
f_2 &\colonequals q - 2q^2 - 3q^3  + 2q^4 - 2q^5 + 6q^6 - q^7 + O(q^9) \in S_2(\Gamma^{(1)}_0(37))^{+} \\
f_3 &\colonequals q + q^3 - 2q^4 - q^7 - 2q^9 + O(q^{11}) \in S_2(\Gamma^{(1)}_0(37))^{-}
\end{aligned}
\end{equation}
be the forms with LMFDB labels \modform{37}{2}{a}{a} and \modform{37}{2}{a}{b}.
In both cases it appears that $\lambda_{p,1} = a_p^2$, where $a_p$ is the $T_p$ eigenvalue of $f_2$ and $f_3$ respectively. Indeed, both are explained in the same fashion by using the transfer map. For example, one can check that $C_0(\phi_2) = f_2 \otimes f_2$, so that \[L_p(\phi_2,T) = L_p(C_0(\phi_2),T,\text{Asai}) = L_p(f_2\otimes f_2,T).\] Comparing linear terms gives $\lambda_{p,1} = a_p^2$.

It remains to explain the eigenform $\phi_4$. It appears that $\lambda_{p,1} = (1+p)a_p$, where $a_p$ is the $T_p$ eigenvalue of $f_3$. This is again explained by the transfer map, since one can check that $C_0(\phi_4) = E_2\otimes f_3$, so that \[L_p(\phi_4,T) = L_p(C_0(\phi_4),T,\text{Asai}) = L_p(E_2\otimes f_3,T).\] Comparing linear terms gives $\lambda_{p,1} = (1+p)a_p$.

Note that $W_{37} f_2 = -f_2$ while $W_{37} E_2 = E_2 $ and $W_{37} f_3 = f_3$, and indeed we only obtain the pairs which are fixed by $W_{37}$ coming from pairs having the same Atkin--Lehner sign (see Remark \ref{rmk:allspinor}).
\end{ex}

\begin{ex}
Let $\Lat$ be a maximal integral lattice with $D = 67^2$ and Gram matrix  \[\left(\begin{array}{cccc}2&0&0&1\\ 0&2&1&0\\ 0&1&34&0\\ 1&0&0&34\end{array}\right).\] Then $\#\Cls(\Lat)= 13$ and we compute the Hecke operator $T_{2,1}$.

Consider the eigenvector $\phi$ satisfying $T_{2,1}\phi = -\phi$. The first few eigenvalues of $\phi$ are: \[\lambda_{2,1} = \lambda_{3,1} = \lambda_{5,1} = \lambda_{7,1} = \lambda_{13} = -1,\quad \lambda_{11,1} = 1,\quad \lambda_{17,1} = 4,\quad \lambda_{19,1} = 29, \quad \dots\] It seems that $\lambda_{p,1} = a_pb_p$, where $a_p$ and $b_p$ are the eigenvalues of the eigenform with LMFDB label \modform{67}{2}{a}{c}: \[ f_1 \colonequals q - \alpha q^2 + (1-\alpha)q^3 (-1+\alpha)q^4 + (1+2\alpha)q^5 + q^6 + \alpha q^7 + O(q^8) \in S_2(\Gamma_0^{(1)}(67)),\] and its Galois conjugate $f_2$, where $\alpha \colonequals (1+\sqrt{5})/2$. This is explained by the transfer map, since $C_0(\phi) = f_1\otimes f_2$, and so \[L_p(\phi,T) = L_p(C_0(\phi), T,\text{Asai}) = L_p(f_1\otimes f_2, T).\] Comparing linear terms gives $\lambda_{p,1} = a_pb_p$.

Note here the necessity of the Galois action appearing in Theorem \ref{thm: Clifford rank four}. We needed to identify $f_1\otimes f_2$ and $f_2\otimes f_1$ in order to uniquely determine $C_0(\phi)$ (similarly with $E_2\otimes f_3$ with $f_3\otimes E_2$). This is clear from an $L$-functions perspective since $L(f_1\otimes f_2, s) = L(f_2\otimes f_1,s)$ and $L(E_2\otimes f_3,s) = L(f_3\otimes E_2, s)$.
\end{ex}

\subsection*{Non-square discriminant}\label{subsub:nsd}

We now consider the somewhat less well-studied nonsquare discriminant case. Here an eigenform $\phi\in M(\Lat)$ can be one of three types.

\begin{ex}
Let $\Lat$ be a maximal integral lattice with $D = 193$, as in Example \ref{ex:rank-4-disc-193}.

The eigenforms $\phi_1,\phi_2,...,\phi_9$ come in three Galois orbits. The eigenvector $\phi_1$ is the Eisenstein eigenvector with eigenvalues \[\lambda_{p,1} = \frac{p^3-1}{p-1} + \chi_{193}(p)p\] and $L$-polynomials \[L_p(\phi_1,T) = (1-\chi_{193}(p)pT)(1-T)(1-pT)(1-p^2T)\] (those of the ``Asai $L$-function" of the Hilbert Eisenstein series $E_2\in M_2(\Z_K)$ over $K = \mathbb{Q}(\sqrt{193})$). 

The seven eigenvectors $\phi_2,...,\phi_8$ appear to have eigenvalues $\lambda_{p,1} = a_{p,i}^2 + p(1-\chi_{193}(p))$ with $a_{p,i}$ running through the $T_p$ eigenvalues of $f \in S_2(\Gamma^{(1)}_0(193),\chi_{193})$ (a Galois orbit of size $14$ with LMFDB label \modform{193}{2}{b}{a}). This is explained by the transfer map. Indeed, we find that $C_0(\phi_2) = \DN(f)$, the Doi--Naganuma lift  \cite{DN}, and so we have
\begin{equation}
\begin{aligned}
L_p(\phi_2,T)(1-\chi_{193}(p)pT) &= L_p(\DN(f), T, \Asai)(1-\chi_{193}(p)pT) \\
&= L_p(f \otimes \bar{f} \otimes \chi_{193}, T)(1-pT).
\end{aligned}
\end{equation}
Comparing linear terms yields $\lambda_{p,1} + \chi_{193}(p) p = a_{p,i}^2 + p$. Note that the Galois action identifies pairs of forms in the orbit, shrinking its size from $14$ to $7$. 

The eigenvector $\phi_9$ is slightly more mysterious. The first few eigenvalues are \[\lambda_{2,1} = -4,\quad \lambda_{3,1} = -4, \quad \lambda_{5,1} = 1, \quad\dots\] These appear to be linked to a Hilbert eigenform $f\in S_2(\Z_K)$. Indeed, there is such an eigenform (LMFDB label \hmf{2.2.193.1}{1}{1}{a}) with Hecke eigenvalues \[a_{\mathfrak{p}_2} = a_{\mathfrak{p}_3} = \frac{1+\sqrt{17}}{2}, \quad a_{\bar{\mathfrak{p}}_2} = a_{\bar{\mathfrak{p}}_3} = \frac{1-\sqrt{17}}{2}, \quad a_{\mathfrak{p}_5} = 1,\quad \dots\] It seems that: \[\lambda_p = \begin{cases} a_{\mathfrak{p}}a_{\bar{\mathfrak{p}}} & \text{ if } p\mathcal{O}_K = \mathfrak{p}\bar{\mathfrak{p}} \text{ splits in } K\\ a_{\mathfrak{p}} & \text{ if } p\mathcal{O}_K=\mathfrak{p} \text{ is inert in } K\end{cases}. \] This follows from the transfer map, since $C_0(\phi_9) = f$ and the above is exactly the linear term of $L_p(f,T,\text{Asai})$.
\end{ex}

\section{Theta series and a theorem of Rallis}\label{sub:theta-rallis}

In the interest of finding explicit formulae for the eigenvalues $\lambda_{p,k}$, we will find it very useful to consider theta series, defined as follows.

First, given a lattice $\Lambda$ of rank $n$ defining the space $M(\Lambda)$ of orthogonal modular forms, we define the \defi{theta map} for $g \in \Z_{\geq 1}$ by 
\begin{equation} \label{eqn:thetadef}
\begin{aligned}\theta^{(g)} \colon M(\Lat) &\to M_{\frac{n}{2}}(\Gamma_0^{(g)}(D),\chi_{D^*})\\ [c_1,...,c_h] &\mapsto \sum_{i=1}^h \frac{c_i}{\#\OO(\Lat_i)}\theta^{(g)}(\Lat_i),\end{aligned}
\end{equation}
where 
\begin{equation}
\theta^{(g)}(\Lat_i)(\tau) \colonequals \sum_{A\in \Mat{n,g}(\Z)}e^{\pi i \tr(A^\textsf{T}Q_i A \tau)}
\end{equation}
is the Siegel theta series of $\Lat_i$ of genus $g$ (with variable in the Siegel upper half plane $\mathcal{H}_g = \{\tau\in M_g(\C)\,|\,\tau^{T} = \tau, \text{Im}(\tau)>0\}$). Here $Q_i$ is the Gram matrix of $\Lambda_i$ with respect to $Q$ and ${}^{\textsf{T}}$ denotes matrix transpose.  By convention, we also define 
\begin{equation}
\theta^{(0)}(\Lat_i) \colonequals 1
\end{equation}
so that 
\begin{equation}
\theta^{(0)}([c_1,\dots,c_h]) = \sum_{i=1}^h \frac{c_i}{\#\OO(\Lat_i)} = \langle [c_1,\dots,c_h], [1,1,\dots, 1]\rangle
\end{equation}
with inner product as in \eqref{eqn:innerprod}.

A long-standing problem has been to determine relations (and non-relations) between Siegel theta series of lattices. For example the fact that $\theta^{(1)}(E_8\oplus E_8) = \theta^{(1)}(E_{16})$ shows that there exist isospectral tori that are non-isometric. The fact that $\theta^{(4)}(E_8\oplus E_8) - \theta^{(4)}(E_{16}) \ne 0$ 
is related to the famous Schottky problem (this function is non-vanishing precisely when $\tau\in\mathcal{H}_4$ corresponds to a $4$-dimensional abelian variety that is not the Jacobian of a genus $4$ curve).

\begin{defn}\label{def:depth} Let $\phi\in M(\Lat)$ be an eigenform. The \defi{depth} $d_\phi$ is the smallest integer such that $\theta^{(d_\phi)}(\phi) \ne 0$.
\end{defn}

In fact $\theta^{(g)}(\phi) \neq 0$ for all $g\geq d_\phi$, since theta series are compatible under the Siegel operator \[\Phi_g \colon M_k(\Gamma_0^{(g)}(D),\chi_{D^*}) \rightarrow M_k(\Gamma_0^{(g-1)}(D),\chi_{D^*}),\] i.e., $\Phi_g(\theta^{(g)}(\Lat)) = \theta^{(g-1)}(\Lat)$ for any lattice $\Lat$ (see B\"ocherer \cite{Bocherer}).

For $p\nmid D$, results of Rallis relate the action of $p^k$-neighbour operators on eigenforms $\phi\in M(\Lat)$ with the action of Hecke operators at $p$ acting on the Siegel modular form $F = \theta^{(g)}(\phi)$ (if non-zero). This implies precise statements relating the Hecke eigenvalues of $\phi$ and $F$. The following is a consequence of such results that will prove useful later. 

\begin{thm}\label{RallisThm}
Let $\phi\in M(\Lat)$ be an eigenform. Suppose that $g\geq 0$ is such that $F \colonequals \theta^{(g)}(\phi)$ has $F \neq 0$.  Let $m \colonequals n/2-1$.  Then the following statements hold.

\begin{enumalph}
\item $F$ is an eigenform for the algebra of Hecke operators generated by $T_p$ when $\chi_{D^*}(p)=1$ and $T_{1,p^2}$ when $\chi_{D^*}(p)=-1$. 
\item If $2g < n$ and $p\nmid D$ then \[L_p(\phi,T) = L_p\left(\chi_{D^*}\otimes F,\,p^{m}T, \emph{ std}\right)\prod_{i=g-m}^{m-g}\left(1 - p^{m - i}T\right).\]
\item If $2g \geq n$ and $p\nmid D$ then \[L_p\left(\chi_{D^*}\otimes F,\,p^{m}T, \emph{ std}\right) = L_p(\phi,T)\prod_{i=(m+1)-g}^{g - (m+1)}\left(1 - p^{m - i} T\right).\]
\end{enumalph}
\end{thm}

Here the standard $L$-function of $\chi\otimes F$ for eigenform $F\in S_k(\Gamma_0^{(g)}(D),\chi)$ has $L$-polynomials \[L_p(\chi\otimes F,\,T,\text{ std}) = (1-\chi(p)T)\prod_{i=1}^g(1-\chi(p)\alpha_i T)(1-\chi(p) \alpha_i^{-1} T),\] at $p\nmid D$, where $\{\alpha_{0,p}, \alpha_{1,p}, ..., \alpha_{g,p}\}$ are the (standard) Satake parameters of $F$ at $p$, normalized so that $\alpha_{0,p}^2\alpha_{1,p}\dots\alpha_{g,p} = 1$. See Pitale \cite{Pitale}*{Chapter 3} for a more detailed discussion. 

\begin{proof}
Parts (b) and (c) follow from work of Rallis \cite{Rallis}*{Remark 4.4}. Part (a) uses an additional Eichler commutation relation from the work of Freitag \cite{Freitag}*{Theorem 4.5} (see also Chenevier--Lannes \cite{cl}*{p.~178, (ii); (7.1.1)}.  The translation in the unimodular case is given explicitly by Chenevier--Lannes \cite{cl}*{Corollary 7.1.3}, but the argument applies more generally, by carefully 
following arrows \cite{Rallis}*{Theorem 6.1}. 
\end{proof}


Note that in the above theorem $F$ may be a lift and so the standard $L$-function may decompose further into $L$-functions corresponding to eigenforms of lower genus. The rank $4$ examples in the previous section already demonstrate this behaviour.

In fact, a consequence of general conjectures of Arthur (known in this case by work of Ta\"ibi \cite{Taibi}) is that the global $L$-function $L(\phi,s)$ should always decompose into a product of automorphic $L$-functions for general linear groups. Knowing this decomposition is related to understanding how $\phi$ is an endoscopic lift, and it lets us understand exactly how the $\lambda_{p,k}$ can be rewritten in terms of eigenvalues of automorphic forms of lower rank groups.

\subsection*{Small depth}\label{subsec:small-depth}

Theorem \ref{RallisThm} tells us that the underlying structure of the eigenvalues $\lambda_{p,k}$ of an eigenform $\phi\in M(\Lat)$ is intimately related to its depth $d_\phi$. 
We begin with small depths; in this case, general formulae can be proved.

\begin{thm}\label{depth01}
Let $\phi\in M(\Lat)$ be an eigenform and let $p\nmid D$ be prime.  Then the following statements hold.
\begin{enumalph}
\item{$d_\phi = 0$ if and only if $\phi$ is the Eisenstein eigenform. In this case: \[L_p(\phi,T) = \left(1-\chi_{D^*}(p)p^{\frac{n}{2}-1}T\right)\prod_{i=0}^{n-2}(1-p^iT)\] and so \[\lambda_{p,1} = \left(\frac{p^{n-1}-1}{p-1}\right) + \chi_{D^*}(p)p^{\frac{n}{2}-1}.\]}
\item{If $d_\phi = 1$ and $F = \theta^{(1)}(\phi)\in S_{\frac{n}{2}}(\Gamma_0^{(1)}(D),\chi_{D^*})$ then: \[L_p(\phi,T) = L_p(\chi_{D^*}\otimes \emph{Sym}^2(F),T)\prod_{i=1}^{n-3}(1-p^iT)\] and so \[\lambda_{p,1} = a_p^2 - \chi_{D^*}(p)p^{\frac{n}{2}-1} + p\left(\frac{p^{n-3}-1}{p-1}\right),\] where $a_p$ is the $T_p$ eigenvalue of $F$.}
\end{enumalph}
\end{thm}

\begin{proof}
If $\phi=[1,1,...,1]$ is the Eisenstein eigenform then $d_\phi = 0$ since \[\theta^{(0)}(\phi) = \text{mass}(\Lat) = \sum_{i=1}^{h}\frac{1}{\#\OO(\Lambda_i)} > 0.\] Conversely, if $\phi$ is not Eisenstein then $d_\phi > 0$ since $\theta^{(0)}(\phi) = 0$ (by definition of the cuspidal subspace).
The Euler factors for the Eisenstein eigenform immediately follow from Theorem \ref{RallisThm}, and the corresponding formula for $\lambda_{p,1}$ follows by comparing linear terms. This formula is expected since the right-hand side is the total number of $p$-neighbours of $\Lat$.

We next consider depth $d_\phi = 1$. To prove the formula for the Euler factor we again use Theorem \ref{RallisThm}. Letting $\{\beta_p, \chi_{D^*}(p)\beta_p^{-1}\}$ be the (spinor) Satake parameters of $F$ at $p$, we find that the (standard) Satake parameters of $F$ at $p$ are $\{\chi_{D^*}(p)\beta_p^2,1, \chi_{D^*}(p)\beta_p^{-2}\}$. These are readily recognised as those of the symmetric square lift $\text{Sym}^2(F)$ of $F$ to $\text{GL}_3$. Hence in this case $L_p(\chi_{D^*}\otimes F, p^{\frac{n}{2}-1}T,\std) = L_p(\chi_{D^*}\otimes \text{Sym}^2(F),T)$, proving the claim. 

Once again, comparing linear terms gives
\begin{equation}
\begin{aligned}\lambda_{p,1} &= p^{\frac{n}{2}-1}(\beta_p^2 + \chi_{D^*}(p) + \beta_p^{-2}) + p\left(\frac{p^{n-3}-1}{p-1}\right)\\ &= p^{\frac{n}{2}-1}((\beta_p+\chi_{D^*}(p)\beta_p^{-1})^2 - \chi_{D^*}(p)) + p\left(\frac{p^{n-3}-1}{p-1}\right) \\ &= a_p^2 - \chi_{D^*}(p)p^{\frac{n}{2}-1} + p\left(\frac{p^{n-3}-1}{p-1}\right)\end{aligned}
\end{equation}
as desired.
\end{proof}

Given the above, it makes sense to focus on finding higher rank lattices that give eigenvectors of higher depth, i.e., $d_\phi \geq 2$. Some of these will still only relate to genus $1$ data (the standard $L$-function of $F$ might break up into $L$-functions of classical modular forms and Dirichlet $L$-functions, e.g.\ if $F$ is an Ikeda lift). However, when the lattice has rank greater than $4$, some will relate to genuine Siegel cusp forms of higher genus, and so are much more mysterious. 

\section{Higher rank}\label{sec:rank-6-8}


Our investigations of orthogonal modular forms also have applications to lattices of rank greater than $4$.  One such application is to find formulas analogous to those of Chenevier--Lannes \cite{cl}*{Th\'eor\`eme A}, expressing the number of ways that the two even unimodular lattices of rank $16$ are $p$-neighbours of each other in terms of the coefficients of $\tau(p)$ and powers of $p$.  There is no genus of even unimodular lattices of order greater than $1$ in rank less than $16$, so we considered lattices of small discriminant.  It is also interesting to study the set of possible types of automorphic forms and their $L$-functions that arise in a given weight and to try to realize them all on specific genera.  In some cases this gives a method for computing Hecke eigenvalues of Siegel modular forms, though not a systematic one since we cannot necessarily produce a genus of lattices corresponding to a given form. 

\subsection*{Rank \texorpdfstring{6}{6}}

 We begin our study of lattices of rank $6$ with a typical small example. Once again, $p$ is assumed to be a good prime, i.e., $p\nmid D$.

\begin{ex} Consider the genus of integral lattices of discriminant $D=39$ as in Example \ref{ex:rank-6-disc-39}.
 The two eigenforms are explained by Theorem \ref{depth01}. The Eisenstein eigenform $\phi_1 = [1,1]$ satisfies $d_{\phi_1} = 0$ and has $L$-polynomials \[L_p(\phi_1,T) = (1-\chi_{-39}(p)p^2T)(1-T)(1-pT)(1-p^2T)(1-p^3T)(1-p^4T)\] and so \[\lambda_{p,1} = \left(\frac{p^5-1}{p-1}\right) + \chi_{-39}(p) p^2.\]
 
 The other eigenform $\phi_2 = [6,-5]$ satisfies $d_{\phi_2} = 1$ and has $L$-polynomials \[L_p(\phi_2,T) = L_p(\chi_{-39}\otimes \Sym^2(f),T)(1-pT)(1-p^2T)(1-p^3T)\] and so \[\lambda_{p,1} = a_p^2 - \chi_{-39}(p)p^2 + p\left(\frac{p^3-1}{p-1}\right)\] (where $a_p$ are the $T_p$ eigenvalues of the newform $f\in S_3(\Gamma_0^{(1)}(39),\chi_{-39})$ with LMFDB label
\modform{39}{3}{d}{c}).

Here, the map $\thet{1}$ is injective and so only classical modular forms contribute to the eigenvalues.
\end{ex}

\begin{ex}\label{ex: gen2lift1} Consider the genus of integral lattices of
  discriminant $D = 75$ that contains the lattice $A_4 \oplus \Lambda_{15}$, where $\Lambda_{15}$ is a lattice of rank $2$ spanned by vectors $x, y$ of norm $4$ with $(x,y) = 1$, as in Example \ref{ex:rank-6-disc-75}. 
Two eigenforms are explained by Theorem $\ref{depth01}$. The Eisenstein eigenform $\phi_1 = [1,1,1]$ is as in the above example (but with character $\chi_{-75}$). The cusp form $\phi_1 = [12,5,-9]$ has depth $d_{\phi_2} = 1$ and so is also as in the above example (but with character $\chi_{-75}$ and modular form $f$ with LMFDB label \modform{75}{3}{c}{e}).

The eigenform $\phi_3 = [16,-10,3]$ has depth $d_\phi = 2$ and generates the kernel of $\theta^{(1)}$. Computation suggests that \[\lambda_{p,1} = (p+1)b_p + (1+\chi_{-3}(p)) p^2,\] where the $b_p$ are
the $T_p$ eigenvalues of the modular form $f\in S_4(\Gamma_0(5))$ with LMFDB label \modform{5}{4}{a}{a}. This may be proved using Theorem \ref{RallisThm}. We know that \[L_p(\phi_3,T) = L_p(\chi_{-3}\otimes F, p^2T,\std)(1-p^2T),\] where $F = \theta^{(2)}(\phi_3)\in S_3(\Gamma_0^{(2)}(75),\chi_{-3})$.  There is no obvious theoretical reason for $F$ to be a lift.
However, our algorithm readily computes that \[L_2(\phi_3,T) = (1-4T)(1+4T)(1+4T+8T^2)(1+8T+32T^2),\] immediately suggesting that $F$ is a lift. We computed more $L$-polynomials and found that they factor in the same way. In fact, it can be shown 
that $F$ is the Ikeda lift of $f$ to $\text{Sp}_4$ (i.e., the Saito--Kurokawa lift of $f$), so that \[L_p(\phi_3,T) = (1-p^2T)(1-\chi_{-3}(p)p^2T)L_p(\chi_{-3}\otimes f,T)L_p(\chi_{-3}\otimes f,pT).\] The formula for $\lambda_{p,1}$ follows by comparing linear terms. 

This example shows that higher depth eigenforms can still have eigenvalues explained by classical modular forms (e.g.\ $\theta^{(d_\phi)}(\phi)$ could be an Ikeda lift, so that $L(\phi,s)$ is a product of $\text{GL}_2$ and Dirichlet $L$-functions).
\end{ex}

\begin{ex}\label{ex:rank-6-disc-84-lift}
There appear to be other eigensystems that are explainable by classical modular forms, but not corresponding to Ikeda lifts. For example, consider the genus of integral lattices with discriminant $D = 84$ that contains $\Lat = A_1^2 \oplus A_2 \oplus L_7$, where $L_7$ is the lattice of rank $2$ and discriminant $7$, as in Example~\ref{ex:rank-6-disc-84}. There is an eigenform $\phi = [2,-8,-2,0,5]$ of depth $d_\phi = 2$.

\begin{conj} The eigenvalues of $\phi$ are:
\[\lambda_{p,1} = pd_p + c_p + (1+\chi_{-3}(p)) p^2,\] where $c_p, d_p$ are the $T_p$ eigenvalues of
modular forms $f_1$ and $f_2$ with LMFDB labels \modform{4}{5}{b}{a} and \modform{7}{3}{b}{a} respectively.  

Alternatively, the eigenvalues are $$\lambda_{p,1} = pe_p^2 + c_p + (\chi_{-3}(p) - \chi_{-7}(p))p^2,$$ where $e_p$ are the coefficients of the modular form $f_3$ with LMFDB label \modform{49}{2}{a}{a}.
\end{conj}

We have checked the above formulae for sufficiently many primes to be convinced that they are correct (all $p < 40$). However, we were unable to give a complete proof. The nature of the second formula suggests that $F = \theta^{(2)}(\phi)$ is a Miyawaki-style lift of $f_1$ and $f_3$ (this would explain the appearance of both a classical eigenform and a symmetric square). We were unable to verify this directly.
\end{ex}

In general, higher depth eigenforms are likely to have eigenvalues that are not completely explained by classical modular forms. The $L$-function cannot be expected to always factor into degree $1$ or $2$ pieces. Eventually we must see a contribution from higher genus Siegel modular forms.

\begin{ex}\label{ex:rank-6-nonlift}
Let us consider the genus of lattices of rank $6$ and discriminant $131$, as in Example~\ref{ex:rank-6-disc-131}.  One such lattice is obtained by adjoining a vector of norm $34$ to $D_5$ that pairs to $1$ with one of the two roots corresponding to a leaf of the Dynkin diagram that is adjacent to the vertex of degree $3$ and to $0$ with the other generators.  The kernel of the $\thet{1}$ map has dimension $1$.  In just a few seconds we use our algorithm to compute the $L$-polynomials for the corresponding eigenform $\phi$ at small primes $p = 3,5,7$ and find irreducible factors of degree $4$; for example
$$L_3(\phi,T) = (1-9T)^2 (1 + 14T + 138T^2 + 1134T^3 + 6561T^4).$$ This shows that $F = \theta^{(2)}(\phi)\in S_3(\Gamma_0^{(2)}(131),\chi_{-131})$ is a non-lift. This is  hard to check directly, since there are currently no general algorithms that allow us to compute $F$. One linear factor is explained by the zeta factor in Theorem \ref{RallisThm}, whereas the other degree $5$ piece corresponds to $L_p(\chi_{-131}\otimes F, p^2T,\std)$. We conclude that \[L_p(\phi,T) = L_p(\chi_{-131}\otimes F, p^2T,\std)(1-p^2T),\] and that $F$ contributes genuine genus $2$ data to the eigenvalues $\lambda_{p,k}$ of $\phi$ (i.e., via the standard lift to a $\text{GL}_5$ automorphic form). In particular, a comparison of linear terms gives: \[\lambda_{p,1} = b_{1,p^2} + p + p^2,\] for $p\neq 131$, where $b_{1,p^2}$ is the $T_{1,p^2}$ eigenvalue of $F$.
\end{ex}

Let us now consider $\thet{2}$; for simplicity we restrict to prime
discriminant.
It is not easy to be certain that an eigenvector is in the kernel of
$\thet{2}$, since there is no Sturm bound $b_p$ to tell us when we can conclude
that a linear combination is $0$ from the first $b_p$ terms being all $0$.
However, it seems reasonable to expect that if the kernel on the coefficients of the lattices of rank $2$ of smallest discriminant is stable and nontrivial over a substantial range, then this is genuinely the kernel of $\thet{2}$.  Under this assumption, we found that $\thet{2}$ appeared not to be injective for the genus of lattices of discriminant $599$, and for $19$ of the $30$ primes congruent to $3 \bmod 4$ between $600$ and $1000$.
(It is also not injective for a number of composite values, the smallest being $471$, but we made no systematic attempt to list these.)

\begin{rem}\label{rem:theta-2-wt-6}
  We might expect that $\thet{2}$ would almost always be injective for 
  $n = 6$ and for only finitely many genera in larger even rank.
Indeed, the mass of a lattice of rank $n$ and discriminant
$p$ grows like $p^{(n-1)/2}$ (this follows easily from the mass
formula of \cite{cs}), and thus the number of lattices in the genus is of this order.  On the other hand, the codomain of $\thet{2}$ is a space of 
modular forms for a group of index roughly $p^3$ (for more detail see
\cite{f}*{Corollary II.6.10}), and so its dimension is proportional to
$p^3$.  Thus the case of rank $6$ is interesting, since $\thet{2}$ fails
to be injective for many $p$ even though there is no obvious reason for this.
\end{rem}

Some investigation of these lattice genera led us to Conjecture \ref{conj:wt-6-ker-theta-2-0}, which we restate here for convenience.

\begin{conj}\label{conj:wt-6-ker-theta-2}
  Let $G_p$ be the genus of lattices of rank $6$ and discriminant $p$
  and let $k_p$ be the
  number of isomorphism classes in $G_p$ of lattices with no
  automorphism of determinant $-1$.  Then the kernel of $\thet{2}$ on
  $G_p$ has dimension $k_p$.
\end{conj}

In particular, such a lattice has no vectors of norm $2$, since
reflection in such a vector has determinant $-1$.  This explains why no such lattices exist for small $p$.

\begin{ex}
We consider the root lattice $A_6$, of rank $6$ and discriminant $7$, as in Example~\ref{ex:rank-6-disc-7}. In trivial weight it only admits an Eisenstein series, but in weight $(4,0,0)$ we find a cusp form $\phi$. We find that for the first few primes
$$
L_p(\phi, T) = L_p(\chi_{-7} \otimes \Sym^2(f), T) (1-p^5 T)(1 - p^6 T)(1 - p^7 T)
$$
where $f \in S_7(\Gamma_0^{(1)}(7), \chi_{-7})$ with LMFDB label \modform{7}{7}{b}{b}.
\end{ex}

\begin{ex}
We consider the lattice from Example~\ref{ex:rank-6-disc-39}, now with forms of nontrivial weight $(2,0,0)$.
The space is of dimension  $4$, and it consists of two Galois orbits of eigenforms, of sizes $3$ and $1$. Denote an eigenform from each orbit by $\phi_1$, $\phi_2$, respectively. 

It appears that
$$
L_p(\phi_1, T) = L_p(\chi_{-39} \otimes \Sym^2(f), T)(1-p^3 T)(1 - p^4 T)(1 - p^5 T)
$$
for all $p \neq 3,13$, where $f \in S_5(\Gamma_0^{(1)}(39), \chi_{-39})$ with LMFDB label \modform{39}{5}{d}{d}; this would follow from an extension of the theta map to higher weight and the corresponding extension of Theorem \ref{RallisThm}.

We also find for $p=2,5,7,11$ that 
$$
L_p(\phi_2, T) = (1 + a_p T + p^8 T^2)(1 - \chi_{-39}(p) p^4 T)(1+p^4 T)(1-p^4 T)^2 
$$
with $a_2 = 2 \cdot 13, a_5 = 47 \cdot 13, a_7 = 49 \cdot 13, a_{11} = -682 \cdot 13$. 
\end{ex}

\subsection*{Rank \texorpdfstring{8}{8}}


As in the case of rank $6$, for small discriminants, the map $\thet{1}$ is injective and we can easily express the eigenvalues of the Kneser matrices in terms of ordinary modular forms.  

\begin{ex} \label{ex:d21}
Consider the genus of lattices of discriminant $D = 21$ containing $\Lat = A_6 \oplus A_2$, as in Example~\ref{ex:rank-8-disc-21}.
In this case $\#\Cls(\Lat) = 3$.  Aside from the Eisenstein eigenform $\phi_1 = [1,1,1]$ with \[L_p(\phi_1,T) = (1-\chi_{21}(p)p^3T)(1-T)(1-pT)(1-p^2T)(1-p^3T)(1-p^4T)(1-p^5T)(1-p^6T)\] and \[\lambda_{p,1} = \left(\frac{p^7-1}{p-1}\right) + \chi_{21}(p)p^3,\] we have eigenforms $\phi_2  = [7,-15,84]$ and $\phi_3 = [3,-4,-32]$, of depth $d_{\phi_2} = d_{\phi_3} = 1$ with  \[L_p(\phi_i,T) = L_p(\chi_{21}\otimes\text{Sym}^2(f_i),T)(1-pT)(1-p^2T)(1-p^3T)(1-p^4T)(1-p^5T)\] and \[\lambda_{p,1} = a_p^2 - \chi_{21}(p)p^3 + p\left(\frac{p^5-1}{p-1}\right),\] where the $a_p$ are the $T_p$ eigenvalues of newforms $f_1, f_2\in S_4(\Gamma_0^{(1)}(21),\chi_{21})$ with LMFDB labels \modform{21}{4}{c}{a} and \modform{21}{4}{c}{b}.

To illustrate our initial goal for this project, we give explicit expressions for the $p$-neighbour adjacency matrices.  Let $\Lat_1 = \Lat$, and let $\Lat_2, \Lat_3$ be the other lattices in the genus, such that the root sublattices of $\Lat_2$ and $\Lat_3$ are $E_6$ and $E_7$ respectively.  Given the eigenvectors and eigenvalues as above, this amounts to a simple change of basis.  The $p$-neighbour adjacency matrix is $1/1309$ times

\[\begin{pmatrix}
816 &  816& 816\\
476  & 476 & 476\\
17 &   17   &  17 
\end{pmatrix}\lambda_{p,1}^{(1)}
+\begin{pmatrix}
196&  - 420& 2352\\
- 245   &  525 & - 2940 \\
49 &  - 105   &   588
\end{pmatrix}\lambda_{p,1}^{(2)}
+\begin{pmatrix}
297 & - 396 &  - 3168\\
- 231  &  308 & 2464\\
- 66  &   88  &   704
\end{pmatrix}\lambda_{p,1}^{(3)}\]
where the $\lambda_{p,1}^{(i)}$ are the $T_p$ eigenvalues for the respective $\phi_i$ given above. One could do the same for the $p^k$-neighbour matrices using the eigenvalues $\lambda_{p,k}^{(i)}$. Clearly the description in terms of the eigenvectors and eigenvalues is more perspicuous.
\end{ex}

For slightly larger discriminants we again expect to see a nontrivial kernel of $\thet{1}$. At first all such eigensystems are explained by classical modular forms (since once again $\theta^{(d_\phi)}(\phi)$ is a lift from $\text{GL}_2$).  For example, the smallest such discriminant is $D = 36$, the genus in question containing the lattice $A_2^2 \oplus D_4$. For the eigenvector $\phi$ of depth $d_\phi = 2$:  \[\lambda_{p,1} = (p+1)a_p + p^2(p+1)^2,\] where the $a_p$
are the coefficients of the newform $f\in S_6(\Gamma_0^{(1)}(3))$ with LMFDB label \modform{3}{6}{a}{a}. As previously explained, this is expected since $\theta^{(2)}(\phi)$ is the Ikeda lift of $f$ to $\text{Sp}_4$ (i.e., Saito--Kurokawa lift). 

Eventually, we expect to see genuine contributions from non-lift Siegel modular forms.  In studying such examples it is important to be able to compute $L$-polynomials, so our general code is essential, as the following example indicates.

\begin{ex}\label{ex:level-53-first-pass}
As in Example \ref{ex:rank-8-disc-53}, we consider the genus of lattices of discriminant $D = 53$. Here $\#\Cls(\Lat) = 8$.
The Eisenstein eigenform
$\phi_1 = [1,1,1,1,1,1,1,1]$ is as in the above example (but with character $\chi_{53}$). The depth $1$ eigenforms all lie in a (messy) Galois orbit with coefficients in the sextic field defined by
\[ x^6 - 2x^5 - 290x^4 - 388x^3 + 14473x^2 + 11014x - 81256.\]
We let $\phi_2$ be one representative of this orbit. Then the $L$-polynomials and eigenvalues of $\phi_2$ are also as in the above example (but with character $\chi_{53}$ and modular form $f\in S_4(\Gamma_0^{(1)}(53),\chi_{53})$ with LMFDB label \modform{53}{4}{b}{a}). 

The eigenform $\phi_3 = [6,0,-96,-21,-42,-16,0,216]$ has depth $d_{\phi_3} = 2$. However, in contrast with Example \ref{ex: gen2lift1}, the eigenform $F = \theta^{(2)}(\phi_3)\in S_4(\Gamma_0^{(2)}(53),\chi_{53})$ is not a lift. This is non-trivial to verify, since there are currently no general algorithms that would allow us to compute $F$ directly. However, we were able to check this by using our algorithm and computing $L$-polynomials of $\phi_3$: 
\begin{equation}
\begin{aligned}L_2(\phi_3,T) &= (1-16T)(1-8T)(1-4T)(1+8T)\\&\qquad\qquad (1+13T+118T^2+832T^3+4096T^4),\\ L_3(\phi_3,T) &= (1-81T)(1-27T)(1-9T)(1+27T)\\&\qquad\qquad (1-12T-129T^2-8748T^3+531441T^4),\\ L_5(\phi_3,T) &= (1-625T)(1-125T)(1-25T)(1+125T)\\
&\qquad\qquad (1-172T+12885T^2-2687500T^3+244140625T^4),\\ L_7(\phi_3,T) &= (1-2401T)(1-343T)^2(1-49T)\\
&\qquad\qquad (1+690T+288617T^2+81177810T^3+13841287201T^4). \end{aligned}
\end{equation}

Three of the linear factors are the zeta factors in Theorem \ref{RallisThm}, whereas the remaining degree $5$ piece corresponds to $L_p(\chi_{53}\otimes F, p^3T,\std)$. Once again, the presence of an irreducible degree $4$ factor in each case indicates that $F$ is a non-lift. We conclude that \[L_p(\phi_3,T) = L_p(\chi_{53}\otimes F, p^3T,\std)(1-p^2T)(1-p^3T)(1-p^4T),\] and that $F$ contributes genuine genus $2$ data to the eigenvalues $\lambda_{p,k}$ of $\phi_3$ (i.e., via the standard lift to a $\text{GL}_5$ automorphic form). In particular, a comparison of linear terms gives: \[\lambda_{p,1} = b_{1,p^2} + p^3 +p^2\left(\frac{p^3-1}{p-1}\right),\] for $p\neq 53$, where $b_{1,p^2}$ is the $T_{1,p^2}$ eigenvalue of $F$.
\end{ex}

In the rank $8$ case it is possible to see non-lift Siegel modular forms of genus $3$ contributing their eigenvalues. Indeed, we were able to compute a genus of lattices of discriminant $D = 269$ and observe an eigenform of depth $3$ corresponding to a non-lift Siegel eigenform $F$ of genus $3$ (clear after computing the corresponding $L$-polynomials). We were also able to check that this is the first prime discriminant for which this happens.


\subsection*{Higher rank}\label{sec:higher-ranks}

For larger ranks, we expect results similar to those we found for rank $6$ and $8$.  However, the transition to vectors of larger depth happens at smaller discriminant and the Hecke operators become much more difficult to calculate, so only a few examples can be analyzed completely.  We briefly discuss one example in rank $10$ and one in rank $12$.

\begin{ex}\label{ex:rank-10}
For rank $10$ and discriminant $D < 15$, the $\thet{1}$ map is always
injective, with results as predicted in Section \ref{subsec:small-depth}.
To illustrate this, we return to Example \ref{ex:A1011}, with root lattice $A_{10}$.  We find that
\begin{equation}
\begin{aligned}
L_p(\phi_1, T) &= (1-\chi_{-11}(p) p^4 T) \prod_{i=0}^8 (1-p^iT) \\
L_p(\phi_2, T) &= L_p(\chi_{-11} \otimes \Sym^2(f_1), T) \prod_{i=1}^7 (1-p^iT) \\
L_p(\phi_3, T) &= L_p(\chi_{-11} \otimes \Sym^2(f_2), T) \prod_{i=1}^7 (1-p^iT) 
\end{aligned}
\end{equation}
for all $p \neq 11$, where $f_1,f_2 \in S_5(\Gamma_0^{(1)}(11), \chi_{-11})$ are the forms with LMFDB labels \modform{11}{5}{b}{a} and  \modform{11}{5}{b}{b}, respectively.
\end{ex}

\begin{ex}
We now consider the genus $G$ consisting of the three lattices $D_4 \oplus D_6$, $D_8 \oplus A_1^2, E_7 \oplus A_1^3$.  There are three rational eigenvectors for $G$; in this ordering they are $[1,1,1]$, $[6,-4,-9]$, $[-12,-56,63]$, of depth $0, 1, 2$ respectively.  We already know how to describe the eigenvalues for the first two of these (the modular form whose symmetric square arises in the second has LMFDB label \modform{4}{5}{b}{a}).  For the third, the Siegel modular form
can once again be shown to be an Ikeda lift to $\text{Sp}_4$ (i.e., a Saito--Kurokawa lift), implying the following formula for eigenvalues:
\begin{equation} \label{eqn:p1ap4m}
\lambda_{p,1} =  (p+1)a_p + \chi_{-4}(p) p^4 + p^2\left(\frac{p^5-1}{p-1}\right),
\end{equation}
where the $a_p$ are the eigenvalues for the modular form with LMFDB label \textup{\modform{2}{8}{a}{a}}.
\end{ex}

\begin{ex}
Consider the genus of lattices of rank $10$ and discriminant $27$ that contains $E_6 \oplus A_2^2$, as in Example~\ref{ex:rank-10-disc-27}.  This is not a maximal lattice, and there are $2$ lattices in the genus, whose theta series are equal; in other words, $[1,-1]$ is an eigenform of depth $2$.  Its eigenvalues are of the form $(p+1)a_p + \chi_{-3}(p) p^4 + \sum_{i=2}^6 p^i$ where the $a_p$ come from the form with LMFDB label \modform{3}{8}{a}{a} (checked for $p = 2, 5, 7$), like \eqref{eqn:p1ap4m}.
\end{ex}

\begin{ex}\label{ex-rank-12}
We proceed to an example of rank $12$ and discriminant $16$: the genus of $D_6 \oplus D_6$.  It contains three other root lattices $E_7 \oplus D_4 \oplus A_1$, $D_{10} \oplus A_1 \oplus A_1$, and $E_8 \oplus A_1^4$, as well as a lattice containing $D_8 \oplus A_1^4$ with index $2$.  The kernel of $\thet{1}$ has dimension $3$; the kernel of $\thet{2}$ appears to have dimension $1$. The eigenvector of depth $1$ has eigenvalues 
\[\lambda_{p,1} = a_p^2 + p^5 +  p\left(\frac{p^9-1}{p-1}\right),\] 
where the $a_p$ come from eigenform with LMFDB label \modform{4}{6}{a}{a}.  
The eigenvectors of depth $2$ are both explained by an Ikeda lift to $\text{Sp}_4$, (i.e., Saito--Kurokawa lift) and so has eigenvalues \[\lambda_{p,1} = (p+1)a_p + p^5 + p^2\left(\frac{p^7-1}{p-1}\right),\] where the $a_p$ come from the eigenform with LMFDB label \modform{2}{10}{a}{a}.  The reason for the duplication is that the lattices in this genus are not maximal: for example, $E_7 \oplus D_4 \oplus A_1$ is a sublattice of index $2$ in $E_8$ (indeed, $E_7 + A_1 \cong \langle r \rangle \oplus r^\perp$, where $r$ is a root of $E_8$), and thus the forms for the genus that is represented by $E_8 \oplus D_4$ and $D_{12}$ appear twice here as well.  The situation is analogous to that of classical modular forms where the forms of weight $k$ and level $N$ appear twice in the space of modular forms of weight $k$ and level $pN$.
The eigenvector of depth greater than $2$ is unexplained and presumably has eigenvalues arising from a non-lift.
\end{ex}

%

\section{Eisenstein congruences}\label{sec:congruences}

One can often easily prove explicit congruences between the eigenvalues $\lambda_{p,k}$ of eigenforms $\phi\in M(\Lat)$. If these eigenforms are explicitly understood as endoscopic lifts (e.g.\ via their $L$-function) then this implies congruences between Hecke eigenvalues of eigenforms for lower rank groups. This has recently been a fruitful strategy when trying to prove non-trivial Eisenstein congruences. 

\begin{ex} \label{exm:sps128}
For a well-known example, consider the $16$-dimensional even unimodular lattice $\Lat = E_8\oplus E_8$ (with the standard inner product). Then famously $\Cls(\Lat) = \{[E_8\oplus E_8], [E_{16}]\}$ and one can calculate in a few seconds that: \[T_{2,1} = \left(\begin{array}{cc}20025 &18225\\ 12870 & 14670\end{array}\right)\] and diagonalizing produces the two eigenforms $\phi_1 = [1,1]$ and $\phi_2 = [405,-286]$ of $M(\Lat)$ with eigenvalues $\lambda_{2,1} = 32895$ and $\lambda'_{2,1} = 1800$ respectively. It is immediately clear that $\lambda_{2,1} \equiv \lambda'_{2,1} \pmod{691}$. Indeed, the simple fact that $286\phi_1 + \phi_2 \equiv [0,0] \pmod{691}$ implies the congruence $\lambda_{p,k} \equiv \lambda'_{p,k} \pmod{691}$ for all $p$ and $k$. 

The cusp form $\phi_2$ has depth $d_{\phi_2} = 4$ and $F = \theta^{(4)}(\phi_2)\in S_8(\text{Sp}_8(\mathbb{Z}))$ can be shown to be the Ikeda lift of $\Delta\in S_{12}(\SL_2(\Z))$ to $\text{Sp}_8$ (see Section $7.3$ of \cite{cl}). It follows that: \[L_p(\chi_D\otimes F, p^7T,\std) = (1-p^7T)L_p(\Delta,T)L_p(\Delta,pT)L_p(\Delta,p^2T)L_p(\Delta,p^3T),\] and so by Theorem \ref{RallisThm} we have the explicit formula (for all $p$):
\[\lambda'_{p,1} = \tau(p)\left(\frac{p^4-1}{p-1}\right) + p^7 + p^4\left(\frac{p^7-1}{p-1}\right).\] Since $\lambda_{p,1} = \frac{p^{15}-1}{p-1} + p^7$ the above congruence reduces to (for $k=1$): \[\left(\frac{p^4-1}{p-1}\right)\tau(p) \equiv \left(\frac{p^4-1}{p-1}\right)(1+p^{11}) \pmod{691}.\]

This is a scaling of the familiar Ramanujan congruence (the scaling factor can be removed by deeper work with the associated Galois representations). The existence of such Eisenstein congruences was used by Ribet in his proof of the converse to Herbrand's theorem \cite{ribet}, relating divisibility of special values of $\zeta(s)$ (i.e., Bernoulli numbers) with the Galois module structure of class groups of cyclotomic fields. In particular, Ramanujan's congruence relates the fact that $\text{ord}_{691}(B_{12}) > 0$ with the fact that $\Cl(\Q(\zeta_{691}))$ has an element of order $691$ satisfying $\sigma\cdot [\mathfrak{a}] = \chi_{691}^{-11}(\sigma)[\mathfrak{a}]$ for all $\sigma\in\text{Gal}(\Q(\zeta_{691})\,|\,\Q)$ (with $\chi_{691}$ the mod 691 cyclotomic character).
\end{ex}

Many other types of $\text{GL}_2$ Eisenstein congruence can be proved by computing orthogonal eigenforms and adopting the above strategy:

\begin{itemize}
\item{Congruences featuring modular forms of non-trivial level/character can be found by considering lattices with non-trivial discriminant. These have moduli dividing special values of Dirichlet $L$-functions.}
\item{Congruences of local origin are also found for lattices of varying discriminant. These have moduli dividing special values of Euler factors \cite{df2}.}
\item{Congruences featuring Hilbert modular forms over totally real fields $F$. Some of these were proven \cite{df} by computing with $\Z_F$-lattices with $F$ a real quadratic field of small discriminant. Some of these congruences were of Ramanujan type (with modulus explained by special values of $\zeta_F(s)$) but others were new and involved non-parallel weight (with modulus explained by special values of adjoint $L$-functions). The second family of congruences were observed experimentally, allowing the authors to make general conjectures.}
\end{itemize}

All such congruences provide evidence for the Bloch--Kato conjecture, a vast generalisation of the Herbrand--Ribet theorem (among other things). This links divisibility of special values of motivic $L$-functions with elements of prescribed order in various Bloch--Kato Selmer groups attached to these motives.

Beyond $\GL_2$ it becomes much harder to prove Eisenstein congruences. Even gaining computational evidence can be tricky, due to the lack of explicit algorithms for computing with higher rank automorphic forms. However, recent interest in computing orthogonal modular forms has led to proofs of non-trivial Eisenstein congruences for higher rank groups.
 
\begin{ex} \label{ex:harder}
A well-known conjecture of Harder suggests that if $j\geq 0$ is even, $k\geq 3$ and $f\in S_{j+2k-2}(\SL_2(\Z))$ is an eigenform, then a (large enough) prime $\mathfrak{q}$ of $\mathbb{Q}_f$ dividing $L_{\text{alg}}(f,j+k)$ should be the modulus of a congruence of the form \[\lambda_{F,p} \equiv \mu_{f,p} + p^{j+k-1}+p^{k-2} \pmod{\frakq},\] for some eigenform $F\in S_{j,k}(\text{Sp}_4(\Z))$ and some $\mathfrak{q}'\mid\mathfrak{q}$ in $\mathbb{Q}_{f,F}$. Typically ``large enough" means that $\mathfrak{q}$ lies above a rational prime $q> j+2k-2$. 

When $j=0$ the right-hand side is the $T_p$ eigenvalue of the Saito--Kurokawa lift of $f$ and so much has been proved. Before the recent work of Chenevier--Lannes, this congruence was unknown for even a single modular form satisfying $j>0$. Their work, featured in the book \cite{cl}, proved the first instance of Harder's conjecture (the case $(j,k) = (4,10)$ and $\mathfrak{q}\mid 41$). This is achieved by proving an explicit congruence between eigenforms in $M(\Lat)$ for $\Lat = E_8\oplus E_8\oplus E_8$ (i.e., $\Cls(\Lat)$ consisting of the $24$-dimensional Niemeier lattices). The congruence then follows once again by a comparison of $L$-functions, although it is now not a simple task to decompose into automorphic $L$-functions.
\end{ex} 

Work of M\'egarban\'e \cite{Megarbane} extended the above to certain high rank lattices of discriminant $2$, leading to proofs of Harder-type congruences (at level $1$). Recent work of Dummigan, Pacetti, Rama and Tornar{\'\i}a \cite{dprt} considers certain quinary lattices and links the eigensystems $\lambda_{p,k}$ with eigenvalues of paramodular forms. Explicit computations led to proofs of Harder-type congruences of paramodular level, as predicted in the paper of Fretwell \cite{Fretwell}.

As expected, some of our own computations have resulted in proofs of new Eisenstein congruences, this time of Kurokawa--Mizumoto type (extending those in \cite{Kurokawa} and \cite{Mizumoto}). We discuss an example.

\begin{ex}\label{ex:disc-53-again} 
Consider the genus of lattices of rank $8$ and discriminant $D=53$. Recall that there is a depth $1$ eigenform $\phi_2 = [6,0,-96,-21,-42,-16,0,216]$ with \[\lambda'_{p,1} = a_p^2 - \chi_{53}(p)p^3 + p\left(\frac{p^5-1}{p-1}\right),\] for $p\neq 53$, where $a_p$ is the $T_p$ eigenvalue of $f = \theta^{(1)}(\phi_3)\in S_4(\Gamma^{(1)}_0(53),\chi_{53})$ (LMFDB label \modform{53}{4}{b}{a}). There is also a rational eigenform $\phi_3$ of depth $2$, with \[\lambda_{p,1} = b_{1,p^2} + p^3 + p^2\left(\frac{p^3-1}{p-1}\right),\] where $b_{1,p^2}$ is the $T_{1,p^2}$ eigenvalue of a non-lift eigenform $F\in S_4(\Gamma_0^{(2)}(53),\chi_{53})$. 

It is possible to normalise the eigenform $\phi_3$ so that the entries are algebraic integers with no common prime ideal factors. After doing this it is then a true splendor to observe that $\phi_2 + 273\phi_1 \equiv [0,...,0] \pmod{\mathfrak{q}}$ for some prime $\mathfrak{q}\mid 397$ of the appropriate sextic number field. This implies a congruence  $\lambda_{p,k} \equiv \lambda'_{p,k} \pmod{\mathfrak{q}}$ for all $p\neq 53$ and $1\leq k \leq 4$. In particular for $k=1$ this reduces to \[b_{1,p^2} \equiv a_p^2 - (1+\chi_{53}(p))p^3 + p + p^5 \pmod{\mathfrak{q}},\] for $p\neq 53$.

It would be interesting to exhibit this Siegel modular form directly, perhaps as a Borcherds product (if it admits such a description).
\end{ex}

Recall that the moduli of Eisenstein congruences should come from special values of $L$-functions. So how do we explain the modulus $\mathfrak{q}$ in the above example? Numerical computations suggest that the norm of 
\[ \frac{L(\text{Sym}^2(f),1)}{\pi^2L(\text{Sym}^2(f),3)} \]
is equal to $24250736770795028/2197125$, which has numerator divisible by $397$. The functional equation would then imply that $\text{ord}_{\mathfrak{q}}(L_{\text{alg}}(\text{Sym}^2(f),6)) > 0$.


As far as the authors are aware the congruence of Example~\ref{ex:disc-53-again} was not previously known, and is intractable using other existing techniques (as with Harder-type congruences). The fact that we were able to prove it is an interesting application of our computations. 

We made similar calculations for all primes congruent to $1 \bmod 4$ and less than $200$ for which $\thet{1}$ has a kernel.  (The discriminant is required to be prime to simplify the determination of the bad factors of the $L$-functions.)  In every case we found a similar divisibility: that is, for all large primes $q$ for which there was a congruence modulo a prime dividing $q$ between a vector in the kernel of $\thet{1}$ and one in the kernel of $\thet{0}$ we found a cusp form $f$ of weight $4$, level $q$, and quadratic character such that $L(\Sym^2(f),1)/\pi^2 L(\Sym^2(f),3)$ appeared to be an algebraic number of norm divisible by $q$.


Likewise, our computations of $L$-functions have been restricted to rank $8$.  It is difficult to perform such computations for genera of rank $6$ because there is only one pair of critical values of the $L$-function, and so we would need another way to calculate the period and find the algebraic part.  However, we have been able to compute $L$-polynomials of Siegel modular forms in this setting and find the predicted congruence.  Our computations suggest the following general conjecture.

\begin{conj}\label{cong}
Let $N\geq 1$ be square-free, $\chi$ be a quadratic character mod $N$, $j\geq 0$ and $k\geq 3$. Suppose $f\in S_{j+k}(\Gamma_0^{(1)}(N),\chi)$ is an eigenform with $\emph{ord}_{\mathfrak{p}}(L_{\emph{alg}}(\emph{Sym}^2(f),j+2k-2)) > 0$ for some prime $\mathfrak{q}$ of $\mathbb{Q}_f$ lying above a rational prime $q > 2(j+k)-1$. Then there exists an eigenform $F\in S_{j,k}(\Gamma^{(2)}_0(N),\chi)$ and a prime $\mathfrak{q}'\mid \mathfrak{q}$ of $\mathbb{Q}_{f,F}$ such that \[b_{1,p^2} \equiv a_p^2 -\chi(p)p^{j+k-1} - p^{j+2k-5} + p^{j+2k-3}+p^{j+1} \pmod{\mathfrak{q}'},\] for all $p\nmid N$ (with $T_p(f) = a_p f$ and $T_{1,p^2}(F) = b_{1,p^2}F$).
\end{conj}

The congruence proved in the example above is the case of $j=0, k=4, N=53$, $\chi = \chi_{53}$ and $\mathfrak{q}\mid 397$ in the number field defining $\phi_3$. This conjecture should hold in higher generality, but for ease of exposition we decided to state it only for square-free level and quadratic character.

Naturally, one asks how Conjecture \ref{cong} fits into the general framework of Eisenstein congruences. The following gives a brief justification.

Just as the right-hand side of Ramanujan-type congruences involve eigenvalues of Eisenstein series, and the right-hand side of Harder-type congruences involve Saito--Kurokawa like eigenvalues, the right-hand side of congruences of Mizumoto-Kurokawa type involve Klingen-Eisenstein like eigenvalues. 

More precisely, suppose that $j\geq 0$ and $k\geq 3$. Then whenever $f\in S_{j+k}(\text{SL}_2(\mathbb{Z}))$ is an eigenform and $\mathfrak{q}$ is a (large enough) prime of $\mathbb{Q}_f$ such that $\text{ord}_{\mathfrak{q}}(L_{\text{alg}}(\text{Sym}^2(f),j+2k-2))>0$, one expects a congruence of the form \[b_p \equiv a_p(1+p^{k-2}) \pmod{\mathfrak{q}'},\] for some eigenform $F\in S_{j,k}(\text{Sp}_4(\mathbb{Z}))$ and some $\mathfrak{q}'\mid\mathfrak{q}$ in $\mathbb{Q}_{f,F}$ (with $T_p(f) = a_p f$ and $T_p(F) = b_pF$). Typically, ``large enough" means that $\mathfrak{q}$ lies above a rational prime $q>2(j+k)-1$.

The right-hand side of this congruence is the eigenvalue of a genus $2$ Klingen-Eisenstein series attached to $f$ (although technically we would need $k>4$ to allow convergence).

Generalisations of Mizumoto-Kurokawa congruences exist for non-trivial level and character. For example, if $f\in S_{j+k}(\Gamma_0^{(1)}(N),\chi)$ satisfies the same divisibility condition then one instead expects a congruence of the form \[b_p \equiv a_p(\chi(p)+p^{k-2})\pmod{\mathfrak{q}'}.\] We are indebted to Neil Dummigan for explaining this to us, as well as explaining how it follows from general conjectures on Eisenstein congruences for split reductive groups in his recent paper with Bergstr\"om \cite{db}*{Section 6}.

Conjecture \ref{cong} is the analogue of this congruence but for the $T_{1,p^2}$ eigenvalues of $F$. To see this, note that the (spinor) Satake parameters at $p\nmid N$ corresponding to the right-hand side of the congruence are $\{\alpha_pp^{\frac{k-2}{2}}, \chi(p)\alpha_pp^{-\frac{k-2}{2}}, \alpha_p^{-1}p^{\frac{k-2}{2}},\chi(p)\alpha_p^{-1}p^{-\frac{k-2}{2}}\}$. Thus the corresponding (standard) Satake parameters are $\{\chi(p)p^{k-2},\chi(p)\alpha_p^2, 1, \chi(p)\alpha_p^{-2}, \chi(p)p^{-(k-2)}\}$. Twisting by $\chi$, scaling suitably and summing gives $a_p^2 - \chi(p)p^{j+k-1} + p^{j+2k-3} + p^{j+1}$. Doing the same with the left-hand side of the congruence (i.e., the (spinor) Satake parameters of $F$ at $p$) gives $b_{1,p^2} + p^{j+2k-5}$. Comparing modulo $\mathfrak{q}'$ reveals the congruence in Conjecture \ref{cong}.

It makes sense that our congruence is the ``standard" version of the original ``spinor" one, since by Theorem \ref{RallisThm} we usually see standard $L$-polynomials of Siegel modular forms appearing in $L_p(\phi,T)$ (as opposed to spinor $L$-polynomials). However, we expect that any $F$ satisfying Conjecture \ref{RallisThm} should also satisfy the original congruence for $T_p$ eigenvalues.

We end by noting that we were also able to observe congruences between eigenforms of higher depths. For example, consider the genus of rank $8$ lattices of genus $D = 269$ mentioned at the end of Section \ref{sec:rank-6-8}. The depth $3$ eigenform corresponds to a non-lift Siegel eigenform $F$ of genus $3$ and there is a depth $2$ eigenform that corresponds to a non-lift Siegel eigenform $F'$ of genus $2$. The two orthogonal eigenforms have eigenvalues that are (provably) congruent mod $\mathfrak{q}\mid 347$. This implies an Eisenstein congruence involving $F$ and $F'$. However, it is not clear what the true explanation of this congruence is in terms of eigenvalues of $F$ and $F'$. We suspect it to be related to a genus $3$ congruence of Mizumoto-Kurokawa type, similar to those discussed in Section $9$ of \cite{db} (but with non-trivial character). Unfortunately, verifying this would require computation of special values of $L(F', s,\std)$. This is far beyond the scope of existing \textsf{Magma} packages.

\end{document}